\documentclass{amsart}
\usepackage[margin=2.5cm]{geometry}
\usepackage{amssymb,amsmath,upgreek,mathrsfs,bm,cases,latexsym,hyperref,color,enumerate,graphicx,algorithm,algorithmic}

\hypersetup{colorlinks,linkcolor={blue},citecolor={blue},urlcolor={blue},hypertexnames=false}

\newtheorem{theorem}{ \bf Theorem}[section]
\newtheorem{proposition}{ \bf Proposition}[section]
\newtheorem{lemma}{ \bf Lemma}[section]
\newtheorem{remark}{ \bf Remark}[section]
\newtheorem{definition}{ \bf Definition}[section]

\newtheorem{problem}{\bf Problem}
\newtheorem{example}{ \bf Example}[section]

\numberwithin{equation}{section}
\allowdisplaybreaks[4]
\begin{document}

	\title[Sensitivity and Robust Control of Coupled Inclusions]{Sensitivity Analysis and Robust Optimal Control for Coupled Evolution Inclusions with State-Dependent Maximal Monotone Operators}

\author[J. Du]{Jinsheng Du}
\address[J. Du]{School of Mathematical Sciences, Chongqing Normal University, Chongqing 401331, China}
\email{jinshengdu@foxmail.com}

\author[B. Mordukhovich]{Boris Mordukhovich}
\address[B. Mordukhovich]{Wayne State University, Detroit, Michigan 48202}
\email{aa1086@wayne.edu}

		\author[S. Zeng]{Shengda Zeng}
\address[S. Zeng]{National Center for Applied Mathematics in Chongqing\\ School of Mathematical Sciences, Chongqing Normal University, Chongqing 401331, China}
\email[Corresponding author]{zengshengda@163.com}
	
	\begin{abstract}
		We consider a class of strongly coupled nonsmooth systems consisting of a semilinear evolution inclusion and a differential inclusion governed by state-dependent maximal monotone operators. Our main contributions are fourfold. First, we collect the well-posedness, compactness, and Painlev\'e--Kuratowski continuity properties of the parameterized solution map required for the subsequent optimization analysis. Second, for Bolza-type optimization over the solution set, we prove the existence of optimal pairs, establish continuity properties of the value function, and derive upper semicontinuity of the optimal-solution map. Third, we study fixed-parameter optimal control, simultaneous control-parameter design, min--max robust control, and Hurwicz-type compromise control under parameter uncertainty, and we establish existence results for each formulation. Fourth, we report numerical experiments for sweeping-type systems that illustrate the sensitivity and robustness phenomena predicted by the theory.
	\end{abstract}


	\maketitle

	\section{Introduction}\label{s1}

	We develop an optimization theory for a class of strongly coupled nonsmooth dynamical systems in which a semilinear evolution inclusion is coupled with a differential inclusion driven by a state-dependent maximal monotone operator. Let $E$ and $H$ be real separable Hilbert spaces and let $I := [0, b]$ ($b > 0$). The basic model takes the following form.
	
	\begin{problem}\label{p1}
		Find $x: I \to E$ and $u: I \to H$ satisfying
		\begin{equation}\label{e1.01}
			\left\{\begin{array}{lll}
				\dot{x}(t) \in A x(t) + F(t, x(t), u(t)) &\text{ for a.e. } t \in I, \\
				\dot{u}(t) \in -B(t,u(t))(u(t)) + G(t, x(t), u(t)) &\text{ for a.e. } t \in I, \\
				x(0) = x_0,\quad u(0) = u_0 \in D\left(B(0, u_0)\right) .
			\end{array}\right.
		\end{equation}
		Here, $A: D(A) \subset E \to E$ denotes the infinitesimal generator of a $C_0$-semigroup $\{T(t)\}_{t \geq 0}$ in $E$. The mapping $B: I \times H \rightrightarrows H$ is such that, for each $(t,u) \in I \times H$, the operator $B(t,u): D(B(t,u)) \subset H \rightrightarrows H$ is maximal monotone and varies with respect to the pseudo-distance. The perturbations $F: I \times E \times H \rightrightarrows E$ and $G: I \times E \times H \rightrightarrows H$ are allowed to be multivalued.
	\end{problem}
	
	Problem~\ref{p1} subsumes several important classes of nonsmooth dynamical systems. When the set-valued mappings $F$ and $G$ reduce to single-valued functions $f$ and $g$, and the operator $B(t,u)$ is identified with the normal cone $N_{C(t,u)}(\cdot)$ to a closed convex set $C(t,u) \subset H$, the inclusion \eqref{e1.01} specializes to the \emph{state-dependent differential sweeping process}
	\begin{equation}\label{e1.01sp}
	\left\{\begin{array}{lll}
		\dot{x}(t) = A x(t) + f(t, x(t), u(t)), \\[2pt]
		-\dot{u}(t) \in N_{C(t, u(t))}(u(t)) + g(t, x(t), u(t)), \\[2pt]
		x(0) = x_0, \quad u(0) = u_0 \in C(0, u_0).
	\end{array}\right.
	\end{equation}
	Since Moreau's pioneering work \cite{moreau1971rafle,moreau1977evolution}, sweeping processes have become indispensable models for quasi-static evolution problems in friction, elastoplasticity, and granular material flow, where the state-dependent constraint set $C(t,u(t))$ represents geometric constraints or yield surfaces depending on the current configuration; see Kunze and Monteiro Marques \cite{kunze1998parabolic} for foundational well-posedness results.

The state dependence in \eqref{e1.01sp} is essential both for modeling and for control: it represents feedback from the current configuration to the constraint that governs its subsequent motion, rather than prescribing the moving set exogenously. Antil, Arndt, Mordukhovich, Nguyen, and Rautenberg \cite{antil2026optimal} provide a concrete example in granular-material accumulation, where the admissible gradient constraint depends on the evolving pile and the supporting surface, while the supporting surface is itself a control variable. This feedback produces a quasi-variational sweeping process, introduces a fixed-point structure, and can destroy the convexity and sequential closedness properties used by the direct method in standard controlled sweeping problems. In our framework, the same mechanism is encoded by the state-dependent operator $u\mapsto B(t,u)$, or equivalently by $C(t,u)$ in \eqref{e1.01sp}. It is precisely why pseudo-distance control, compactness, and Painlev\'e--Kuratowski continuity of the parameterized solution sets are needed before the joint-design, min--max, and Hurwicz formulations can be analyzed.
	
	More generally, when $B(t,u)$ admits a variational structure, i.e., $B(t,u)=\partial\varphi(t,u;\cdot)$ for a proper, lower semicontinuous, convex functional $\varphi(t,u;\cdot): H\to\mathbb{R}\cup\{+\infty\}$, and $F$, $G$ are single-valued, the second line of \eqref{e1.01} is equivalent to the \emph{differential variational inequality}
	\begin{equation}\label{e1.01vi}
	\left\{\begin{array}{lll}
		\dot{x}(t) = A x(t) + f(t, x(t), u(t)), \\[2pt]
		\langle \dot{u}(t) - g(t, x(t), u(t)),\, v - u(t) \rangle
		+ \varphi(t, u(t); v) - \varphi(t, u(t); u(t)) \geq 0
		\quad \forall\, v \in H, \\[2pt]
		x(0) = x_0, \quad u(0) = u_0 \in \operatorname{dom}(\varphi(0, u_0; \cdot)).
	\end{array}\right.
	\end{equation}
	Such differential variational inequalities arise in nonsmooth contact mechanics, traffic network equilibrium, and electrical circuit systems with ideal diodes; systematic development of DVIs was initiated by Pang and Stewart \cite{pang2008differential}, and has been extended to the infinite-dimensional setting by Liu, Motreanu, and Zeng \cite{liu2018nonlinear} and Mig\'orski, Sofonea, and Zeng \cite{migorski2019well}.
	
	\smallskip
	The well-posedness and Painlev\'{e}--Kuratowski sensitivity analysis for Problem~\ref{p1} have been established in the companion paper \cite{zengduwang2026submitted}. Building on those results, the present work develops the associated optimization theory: we study Bolza-type cost functionals evaluated on the parameterized solution sets of \eqref{e1.01} and, in the controlled setting, on solution sets that additionally depend on external control inputs.
	
	To formulate these questions, let $(\Lambda, d_\Lambda)$ and $(\Xi, d_\Xi)$ be metric parameter spaces associated with perturbations of the $x$-subsystem and the $u$-subsystem, respectively. For each $(\zeta,\eta) \in \Lambda \times \Xi$, we consider the parameterized system
	
	\begin{problem}\label{p2a}
		Find pairs $(x, u)$ with $x: I \to E$ and $u: I \to H$ such that
		\begin{equation}\label{e1.01a}
			\left\{\begin{aligned}
				\dot{x}(t) &\in A_\zeta x(t) + F(t, x(t), u(t), \zeta), && \text{for a.e. } t \in I, \\
				\dot{u}(t) &\in -B(t, u(t), \eta)(u(t)) + G(t, x(t), u(t), \eta), && \text{for a.e. } t \in I, \\
				x(0) &= x_0(\zeta), \\
				u(0) &= u_0(\eta) \in D(B(0, u_0(\eta), \eta)).
			\end{aligned}\right.
		\end{equation}
	\end{problem}
	On the corresponding admissible solution set $\mathcal{S}(\zeta,\eta)$, we study the Bolza-type functional
	\begin{equation}\label{e1.02}
		\mathcal{J}(x, u, \zeta, \eta)
		:= \int_0^b \ell(t, x(t), u(t), \zeta, \eta) \, dt
		+ \varphi(x(b), u(b), \zeta, \eta).
	\end{equation}
	To incorporate an external control, let $Y$ be a real separable Hilbert space, let $\mathcal{B}: Y \to E$ be a bounded linear operator, and let $W: I \rightrightarrows Y$ be a control constraint map. We define the admissible-control set by
	$$
	\mathcal{W}_{ad} := \left\{ w \in L^2(I, Y) \mid w(t) \in W(t) \text{ for a.e. } t \in I \right\}.
	$$
	For $w \in \mathcal{W}_{ad}$ and $(\zeta,\eta) \in \Lambda \times \Xi$, we also consider the controlled system
	\begin{problem}\label{p2b}
		For fixed $w \in \mathcal{W}_{ad}$ and $(\zeta,\eta) \in \Lambda \times \Xi$, find pairs $(x, u)$ with $x: I \to E$ and $u: I \to H$ such that
		\begin{equation}\label{e1.03}
			\left\{\begin{aligned}
				\dot{x}(t) &\in A_\zeta x(t) + F(t, x(t), u(t), \zeta) + \mathcal{B}w(t), && \text{for a.e. } t \in I, \\
				\dot{u}(t) &\in -B(t, u(t), \eta)(u(t)) + G(t, x(t), u(t), \eta), && \text{for a.e. } t \in I, \\
				x(0) &= x_0(\zeta), \\
				u(0) &= u_0(\eta) \in D(B(0, u_0(\eta), \eta)).
			\end{aligned}\right.
		\end{equation}
	\end{problem}
	Problem~\ref{p2a} provides the parameterized feasible-set model for the Bolza problem studied in Section~\ref{s5}, while Problem~\ref{p2b} is the controlled model underlying Section~\ref{s6}. In Section~\ref{s5} we prove continuity of the value function and upper semicontinuity of the optimal-solution map for the Bolza problem on $\mathcal{S}(\zeta,\eta)$; in Section~\ref{s6} we establish existence results for nominal, joint-design, robust, and Hurwicz-type control formulations.
	
	\subsection{Motivating Problem Classes}
	
	The optimization models studied here are motivated by three classical paradigms from decision theory and control design.
	
	\subsubsection{Joint Design and Control Optimization}
	In many applications, the control and the system parameters are not chosen separately; instead, one seeks a coordinated design minimizing a common performance index. This viewpoint originates in the simultaneous system-design framework of Ghosh \cite{ghosh1986simultaneous,ghosh1988simultaneous} and connects to the bilevel-programming perspective of Dempe \cite{dempe2002bilevel} and the Moreau-envelope-based bilevel algorithms of Bai, Zeng, Zhang, and Zhang \cite{bai2026alternating}. In the present framework, the external control $w$ and the structural-parameter pair $(\zeta,\eta)$ are optimized simultaneously over a feasible set generated by the coupled inclusion. Letting $\mathcal{P}\subset\Lambda\times\Xi$ denote a compact admissible design--parameter set, we study the simultaneous problem
	$$
	\inf_{(w,\,\zeta,\,\eta)\,\in\,\mathcal{W}_{ad}\times\mathcal{P}}
	\bigl\{m_{\zeta,\eta}(w)+\mathcal{R}(w)\bigr\},
	$$
	where $m_{\zeta,\eta}(w)$ is the optimal value of the Bolza state--cost over the trajectory set $\mathcal{S}_{\zeta,\eta}(w)$ and $\mathcal{R}$ is a control regularization term; see Problem~\ref{p6} in Section~\ref{s6}.
	
	\subsubsection{Robust Min--Max Control Under Uncertainty}
	When the parameters are not design variables but uncertain quantities ranging over a prescribed uncertainty set, a natural formulation is a worst-case or min--max control problem. This perspective has a long history in control and optimization, encompassing the Hamilton--Jacobi--Bellman approach to minimax control \cite{barron1995relaxed,vinter2005minimax}, the robust optimization paradigm \cite{ben2009robust}, and the $H^\infty$ design theory \cite{basar2008h,zhou1996robust}; recent algorithmic and structural advances can be found in \cite{scagliotti2025minimax,bot2023alternating,tsaknakis2023minimax,xu2024derivative,delfour2025semidifferential,zhang2023sion}; see Subsection~1.2 for a detailed account. In the present paper, this leads to a robust control problem in which the controller is selected against the most adverse admissible parameter realization. Specifically, with the same compact uncertainty set $\mathcal{P}$, we formulate the min--max problem
	$$
	\inf_{w\,\in\,\mathcal{W}_{ad}}
	\sup_{(\zeta,\,\eta)\,\in\,\mathcal{P}}
	\bigl\{m_{\zeta,\eta}(w)+\mathcal{R}(w)\bigr\},
	$$
	where the outer minimization selects the control while the inner maximization captures the worst-case parameter realization; see Problem~\ref{p7} in Section~\ref{s6}.
	
	\subsubsection{Hurwicz-Type Compromise Criteria}
	The Hurwicz criterion interpolates between best-case and worst-case evaluations through an optimism index $\alpha \in [0,1]$, providing a one-parameter family that bridges optimistic design and robust control. Introduced in a decision-theoretic context by Arrow and Hurwicz \cite{arrow1977optimality} and axiomatized by Ghirardato, Maccheroni, and Marinacci \cite{ghirardato2004differentiating}, the criterion reduces for $\alpha=0$ to the max-min expected utility of Gilboa and Schmeidler \cite{gilboa1989maxmin}. Uncertain optimal control under this criterion has been studied in finite-dimensional settings by Sheng, Zhu, and Hamalainen \cite{sheng2013uncertain}, Deng and Shen \cite{deng2021hurwicz}, and Li, Song, Liu, and Alsaadi \cite{li2022nash}. In the present setting, this yields a compromise control model interpolating between optimistic and worst-case parameter selection while preserving the same coupled dynamical feasibility structure. Given $\alpha\in[0,1]$, the Hurwicz objective is
	$$
	\inf_{w\,\in\,\mathcal{W}_{ad}}
	\Bigl\{\alpha\inf_{(\zeta,\eta)\in\mathcal{P}}m_{\zeta,\eta}(w)
	+(1-\alpha)\sup_{(\zeta,\eta)\in\mathcal{P}}m_{\zeta,\eta}(w)
	+\mathcal{R}(w)\Bigr\},
	$$
	so that $\alpha=0$ recovers the robust formulation and $\alpha=1$ reduces to the optimistic design; see Problem~\ref{p8} in Section~\ref{s6}.
	
	\subsection{Related Work and Gaps}
	
	The main analytical difficulty is that the admissible set is itself the solution set of a coupled inclusion whose operator geometry is endogenous. For a comprehensive account of dynamical systems coupled with monotone set-valued operators we refer to the survey by Brogliato and Tanwani \cite{brogliato2020dynamical}. For time-dependent maximal monotone operators, the pseudo-distance introduced by Vladimirov \cite{vladimirov1991nonstationary,vladimirov1990differential} has become a standard tool for controlling moving domains, and related well-posedness and approximation results appear in \cite{azzam2018perturbed,azzam2019perturbed,vilches2020evolution}. In the state-dependent case, the geometry of the monotone part depends on the evolving state and hence on the companion equation. Progress in this direction includes the well-posedness and invariant-set results of Le \cite{le2020well}, the nonsmooth Lyapunov framework of Adly, Hantoute, and Th\'era \cite{adly2016nonsmooth}, and existence results for coupled state-dependent systems obtained by Selamnia, Azzam-Laouir, and Monteiro Marques \cite{selamnia2022evolution} and Benguessoum, Azzam-Laouir, and Castaing \cite{benguessoum2021time}.
	
	On the optimization side, a substantial body of results is now available for nonsmooth sweeping-type systems. Relaxed optimal control for perturbed sweeping processes was studied by Edmond and Thibault \cite{edmond2005relaxation}. A major line of research, initiated by Colombo, Mordukhovich, and Nguyen \cite{colombo2019optimal,colombo2020optimization}, develops existence, discrete approximations, and necessary optimality conditions for controlled sweeping processes; extensions to nonconvex, polyhedral, and integro-differential settings appear in \cite{bouach2022optimal,cao2017optimality,cao2019optimal,henrion2023controlled}. Maximum-principle-based results have been obtained by Pinho, Ferreira, and Smirnov \cite{pinho2023maximum}, and sensitivity results for second-order evolution subdifferential inclusions by Bartosz, Denkowski, and Kalita \cite{bartosz2015sensitivity}.

	A recent study \cite{antil2026optimal} examined optimal control of a quasi-variational sweeping model for granular materials. Using regularization together with space and time discretization, it established the existence of optimal solutions for a time-continuous, space-discrete problem and derived necessary optimality conditions for a fully discrete problem. These results demonstrate both the applied relevance of state-dependent constraints and the additional analytical work they require. The present paper is complementary: instead of a model-specific gradient constraint and a discrete optimality system, we treat abstract infinite-dimensional coupled inclusions governed by state-dependent maximal monotone operators and establish solution-set, value-function, and optimal-set sensitivity, as well as existence for fixed-parameter control, joint design, robust min--max control, and Hurwicz compromise control.
	
	For time-dependent maximal monotone dynamics, Hermosilla and Palladino \cite{hermosilla2022optimal} and Sa\"{\i}di and Yarou \cite{saidi2017control} established existence of optimal controls. In the parameter-dependent setting, Papageorgiou, R\u{a}dulescu, and Repov\v{s} \cite{papageorgiou2017sensitivity} proved Hadamard well-posedness of the associated optimal control problems, Adly and Zakaryan \cite{adly2019sensitivity} treated generalized Bolza and Mayer problems for parametric nonconvex sweeping processes, and Adly and Rockafellar \cite{adly2021sensitivity} developed a second-order sensitivity analysis via proto-differentiability of resolvents. Within the variational-hemivariational framework, Li and Liu \cite{li2018sensitivity}, Zeng, Mig\'orski, and Liu \cite{zeng2021well}, and Zeng, Mig\'orski, and Khan \cite{zeng2021nonlinear} obtained existence, compactness, sensitivity, and optimal-control results.
	
	Despite this progress, the existing literature does not cover the combination of features addressed here, in particular the three optimization paradigms introduced in Subsection~1.1. We highlight the principal gaps.
	
	\smallskip\noindent\textit{Joint design.} The simultaneous system-design literature (e.g., Ghosh \cite{ghosh1986simultaneous,ghosh1988simultaneous}) addresses finite-dimensional plant--controller co-design for linear systems. Recent bilevel optimization algorithms, such as the Moreau envelope--based approach of Bai, Zeng, Zhang, and Zhang \cite{bai2026alternating}, provide efficient gradient-type methods for hierarchical design problems but are restricted to finite-dimensional settings. In the infinite-dimensional setting, the available works optimize either the control for fixed parameters \cite{colombo2020optimization,hermosilla2022optimal} or study parameter sensitivity of a fixed control structure \cite{adly2019sensitivity,papageorgiou2017sensitivity}. No existing framework treats joint optimization of controls and structural parameters when the admissible set is generated by a coupled evolution inclusion with state-dependent maximal monotone operators.
	
	\smallskip\noindent\textit{Robust min--max control.} Classical minimax control theory (Barron and Jensen \cite{barron1995relaxed}, Vinter \cite{vinter2005minimax}) is developed in a Hamilton--Jacobi--Bellman/viscosity-solution setting for ODE dynamics, while the robust optimization framework of Ben-Tal, El Ghaoui, and Nemirovski \cite{ben2009robust} addresses finite-dimensional mathematical programs. The recent sweeping-process and monotone-operator control results cited above all treat \emph{nominal} (fixed-parameter) problems. Minimax sliding-mode designs for linear evolution equations have been developed by Zhuk, Iftime, Epperlein, and Polyakov \cite{zhuk2021minimax}, and minimax control-affine ensemble problems by Scagliotti \cite{scagliotti2025minimax}. Algorithmic advances for nonconvex-concave minimax formulations appear in \cite{bot2023alternating,delfour2025semidifferential,tsaknakis2023minimax,xu2024derivative,zhang2023sion}. All of these results, however, rely on linearity, finite-dimensional structure, or standard ODE dynamics rather than monotone-operator geometry. No existing work provides a worst-case performance guarantee over a parameter uncertainty set for dynamics governed by state-dependent monotone operators.
	
	\smallskip\noindent\textit{Hurwicz compromise.} The Hurwicz criterion \cite{arrow1977optimality} is a classical tool in decision theory under ignorance. It has been applied to uncertain optimal control for ODE systems \cite{sheng2013uncertain,deng2021hurwicz} and to multi-player differential games \cite{li2022nash}, but all existing results assume finite-dimensional ODE dynamics or uncertainty-theory frameworks. To the best of our knowledge, it has not been applied to optimal control of any infinite-dimensional evolution system, and \emph{a fortiori} not to coupled inclusions with parameter-dependent feasibility.
	
	\smallskip
	More fundamentally, in \eqref{e1.01a} the maximal monotone operator depends on the current state, the two subsystems are mutually coupled, and optimization is performed over a parameter-dependent family of solution sets rather than a fixed feasible set. Because of this endogenous geometry, continuity of the value function, upper semicontinuity of optimal pairs, and existence of solutions for all three formulations above do not follow from the available theory for time-dependent operators, classical minimax control, variational-hemivariational inclusions, or standard perturbation analysis of optimization problems \cite{bonnans2000perturbation}.
	
	\subsection{Contributions}
	
	The aim of this paper is to develop an optimization theory for the class of coupled inclusions described above. The well-posedness and sensitivity results collected in the earlier sections serve as prerequisites; the main focus is the parameter-dependent and control-dependent optimization layer developed in Sections~\ref{s5} and~\ref{s6}. Our principal contributions are as follows:
	\begin{itemize}
		\item We formulate parameter-dependent Bolza optimization problems over solution sets of coupled evolution inclusions with state-dependent maximal monotone operators.
		\item We identify the well-posedness, compactness, and Painlev\'e--Kuratowski continuity properties of the solution map that are needed to transfer trajectory stability to optimization stability.
		\item We prove existence of optimal pairs, continuity of the value function, and upper semicontinuity of the optimal-solution map for the parameterized Bolza problem.
		\item We extend the framework to controlled dynamics and establish existence results for fixed-parameter optimal control, simultaneous control-parameter design, min--max robust control, and Hurwicz-type compromise control.
		\item We support the abstract theory with numerical experiments on sweeping-type systems illustrating sensitivity, nominal control, joint design, and robust behavior.
	\end{itemize}
	
\subsection{Notation}\label{s1.5}

	Throughout the paper, $E$ and $H$ denote real separable Hilbert spaces with inner products $\langle\cdot,\cdot\rangle_E$, $\langle\cdot, \cdot\rangle$ and norms $\|\cdot\|_E$, $\|\cdot\|$, respectively, and $I := [0, b]$ ($b > 0$) is a fixed finite interval. We write $Z$ for either $E$ or $H$ as a generic placeholder. The closed balls of radius $r > 0$ in $E$ and $H$ are denoted by $\mathbb{B}_E(x, r)$ and $\mathbb{B}(u, r)$; the unit balls by $\mathbb{B}_E$ and $\mathbb{B}_H$. For $1 \leq p < \infty$, $L^p(I, Z)$ is the space of Bochner integrable functions $w\colon I \to Z$ with norm
	$$
	\|w\|_{L^p(I, Z)} = \left(\int_0^b\|w(s)\|_Z^p\, d s\right)^{1 / p}.
	$$
	$C(I, Z)$ is the space of continuous functions from $I$ to $Z$ equipped with the supremum norm
	$$
	\|w\|_{C(I, Z)} = \sup_{t \in I}\|w(t)\|_Z.
	$$
	The Sobolev space $W^{1,1}(I, Z)$ consists of absolutely continuous functions $w\colon I \to Z$ satisfying
	$$
	w(t) = w(0) + \int_0^t v(s)\, d s \quad \text{for all } t \in I,
	$$
	for some $v \in L^1(I, Z)$, with $\dot{w} = v$ a.e.\ on $I$. For a nonempty closed convex set $S \subset H$, $d_S(u) := \inf_{v \in S}\|u - v\|$ denotes the distance from $u$ to $S$, and $\operatorname{proj}_S(u)$ the metric projection. A subset $S \subset H$ is said to be ball-relatively compact if $S \cap \overline{\mathbb{B}(0,r)}$ is relatively compact in $H$ for every $r > 0$.

\subsection{Outline of the paper}
	The remainder of the paper is organized as follows. Section~\ref{s2} reviews preliminary material on maximal monotone operators, Kuratowski set limits, and multivalued mappings. Section~\ref{s3} collects the standing assumptions and the well-posedness and parameter-sensitivity results needed for the admissible-set analysis. Section~\ref{s5} studies the sensitivity of optimal values and optimal pairs for the Bolza problem \eqref{e1.02}. Section~\ref{s6} develops the external-control layer, including joint design, robust, and Hurwicz compromise formulations. Section~\ref{s10} presents numerical experiments on sweeping-type benchmarks. Finally, Section~\ref{s11} concludes the paper and outlines directions for future work.

	\section{Preliminaries}\label{s2}
	
	This section collects the operator-theoretic and set-valued analysis tools used in the subsequent sections. The general notation and function-space conventions were fixed in Subsection~\ref{s1.5}.
	
	We begin with maximal monotone operators; see \cite{alber2006nonlinear,barbu2010nonlinear,kunze1997bv} for comprehensive accounts. An operator $B: D(B) \subset H \rightrightarrows H$ is monotone if
	$$
	\langle y_1 - y_2, x_1 - x_2 \rangle \geq 0 \quad \text{for all } y_i \in B(x_i), i=1,2.
	$$
	$B$ is maximal monotone if its graph is not properly contained in that of any other monotone operator, or equivalently, if $R(I_H + \lambda B) = H$ for all $\lambda > 0$. If $B$ is maximal monotone, $\overline{D(B)}$ is convex, and for any $x \in D(B)$, the value $B(x)$ is closed and convex, containing a unique minimal-norm element $B^0(x) := \operatorname{proj}_{B(x)}(0)$.
	Given two maximal monotone operators $B_1$ and $B_2$, their pseudo-distance is defined as
	$$
	\operatorname{dis}(B_1, B_2) = \sup_{\substack{x \in D(B_1), y \in B_1(x) \\ x' \in D(B_2), y' \in B_2(x')}} \frac{\langle y - y', x' - x \rangle}{1 + \|y\| + \|y'\|}.
	$$
	This pseudo-distance is symmetric, takes values in $[0, \infty]$, and $\operatorname{dis}(B_1, B_2)=0$ if and only if $B_1 = B_2$.
	
	The following closedness property for sequences of maximal monotone operators will be used later (see \cite{kunze1997bv}).

	\begin{lemma}\label{l2.02}
		Let $B$ and a sequence $\{B_n\}_{n \in \mathbb{N}}$ of maximal monotone operators on $H$ be such that $\operatorname{dis}(B_n, B) \to 0$ as $n \to \infty$. Then:
		\begin{enumerate}[{\rm (i)}]
			\item if $x_n \in D(B_n)$ with $x_n \to x$ strongly in $H$, and $y_n \in B_n(x_n)$ with $y_n \rightharpoonup y$ weakly in $H$, then $x \in D(B)$ and $y \in B(x)$;
			
			\item if there exists a constant $M > 0$ such that $\|B_n^0(x)\| \leq M(1 + \|x\|)$ for all $x \in D(B_n)$ and all $n \in \mathbb{N}$, then for every $y \in D(B)$, there exists a sequence $\{y_n\}_{n \in \mathbb{N}}$ with $y_n \in D(B_n)$ such that $y_n \to y$ and $B_n^0(y_n) \to B^0(y)$ in $H$;
			
			\item if $x \in D(B)$ and $y \in H$ satisfy $\langle B^0(z) - y, z - x \rangle \geq 0$ for all $z \in D(B)$, then $y \in B(x)$.
		\end{enumerate}
	\end{lemma}
	
	For any nonempty bounded subsets $A$ and $B$ of a Banach space $X$, the Hausdorff metric $\mathcal{H}(A, B)$ is defined by
	$$
	\mathcal{H}(A, B) := \max \left\{ \sup_{a \in A} d(a, B), \, \sup_{b \in B} d(b, A) \right\},
	$$
	We also require the concept of topological limits for sequences of sets.
	\begin{definition}\cite[Definition 3.14]{migorski2012nonlinear}
		Let $(X, \tau)$ be a Hausdorff topological space with topology $\tau$, and let $\{A_n\}_{n \in \mathbb{N}}$ be a nonempty sequence of subsets of $X$. We define the lower and upper topological limits as follows:
		$$
		\begin{aligned}
			\tau\text{-}\liminf _{n \rightarrow \infty} A_n &= \left\{y \in X : \exists y_n \in A_n \text{ for all } n \in \mathbb{N} \text{ such that } y_n \xrightarrow{\tau} y \right\}, \\
			\tau\text{-}\limsup _{n \rightarrow \infty} A_n &= \left\{y \in X : \exists \text{ a subsequence } \{n_k\} \text{ and } y_{n_k} \in A_{n_k} \text{ such that } y_{n_k} \xrightarrow{\tau} y \right\}.
		\end{aligned}
		$$
		The set $\tau\text{-}\liminf A_n$ is called the $\tau$-Kuratowski lower limit of the sets $A_n$, and $\tau\text{-}\limsup A_n$ is called the $\tau$-Kuratowski upper limit of the sets $A_n$.
	\end{definition}
	
	We next recall basic definitions concerning multivalued mappings; for further details, see \cite{deimling2011multivalued}.
	
	\begin{definition}
		Let $X$ and $Y$ be Banach spaces. A set-valued map $\Phi: X \rightrightarrows Y$ is said to be:
		\begin{enumerate}[{\rm (i)}]
			\item upper semicontinuous (u.s.c.) if for every open set $\Omega \subset Y$, the set
			$\Phi^{+}(\Omega) := \{x \in X \mid \Phi(x) \subset \Omega\}
			$ is open in $X$;
			\item lower semicontinuous (l.s.c.) if for every open set $\Omega \subset Y$, the set $
			\Phi^{-}(\Omega) := \{x \in X \mid \Phi(x) \cap \Omega \neq \emptyset\}
			$ is open in $X$;
		\end{enumerate}
		
		Furthermore, regarding time-dependent maps on an interval $I:=[0, b]$, we introduce the following properties:
		\begin{enumerate}[{\rm (i)}]
			\setcounter{enumi}{2}
			\item A map $F: I \rightrightarrows X$ is measurable if for every open subset $O \subset X$, the set $F^{+}(O) = \{t \in I \mid F(t) \subset O\}$ is Lebesgue measurable in $\mathbb{R}$.
			
			\item A map $F: I \rightrightarrows X$ with nonempty bounded convex values is strongly measurable if there exists a sequence $\{F_n\}_{n=1}^{\infty}$ of step set-valued maps with bounded convex values such that
			$$
			\mathcal{H}\left(F(t), F_n(t)\right) \rightarrow 0 \quad \text { as } n \rightarrow \infty \quad \text { for a.e. } t \in I.
			$$
			
			\item A map $F: I \times X \rightrightarrows Y$ is superpositionally measurable if for every measurable map $Q: I \rightrightarrows X$ with compact values, the composition $\Phi: I \rightrightarrows Y$ defined by $\Phi(t):=F(t, Q(t))$ is measurable.
			
			\item A map $\Psi: I \times X \rightrightarrows Y$ is integrably bounded if there exists a function $\varrho \in L^1(I, \mathbb{R}_+)$ such that for a.e. $t \in I$ and all $x \in X$,
			$$
			\|\Psi(t,x)\|_Y := \sup_{y \in \Psi(t,x)} \|y\|_Y \le \varrho(t).
			$$
		\end{enumerate}
	\end{definition}

	The following proposition provides equivalent characterizations of upper semicontinuity, which will be crucial for our sensitivity analysis.
	
	\begin{proposition}\cite[Proposition 3.8]{migorski2012nonlinear}\label{p2.01}
		Let $X$ and $Y$ be metric spaces and $\Phi: X \rightrightarrows Y$ be a set-valued map. The following statements are equivalent:
		\begin{enumerate}[{\rm (i)}]
			\item $\Phi$ is upper semicontinuous;
			\item For every closed set $D \subset Y$, the inverse image $\Phi^{-}(D) := \{x \in X \mid \Phi(x) \cap D \neq \emptyset\}$ is closed in $X$;
			\item If $x \in X$, $\{x_n\} \subset X$ with $x_n \to x$, and $V \subset Y$ is an open set such that $\Phi(x) \subset V$, then there exists $n_0 \in \mathbb{N}$, depending on $V$, such that $\Phi(x_n) \subset V$ for all $n \geq n_0$.
		\end{enumerate}
	\end{proposition}
	
\section{Assumptions and Preliminary Results}\label{s3}

To streamline the later sections, we collect in one place the standing hypotheses and preliminary conclusions used throughout the paper. The abstract solvability and sensitivity part of this section is adapted from \cite{zengduwang2026submitted}, whereas the optimization-oriented assumptions on costs and controls are tailored to the present paper. For the convenience of the reader, we proceed in two steps: we first record the assumptions on the parameterized and control-driven dynamics, together with the optimization layer hypotheses, and then gather the well-posedness, compactness, and parameter-stability results that will be used later.

\subsection{Standing Assumptions for the Abstract and Sweeping Control Framework}
	
	We formulate all assumptions directly for the parameterized system in Problem \ref{p2a} and the control-driven system in Problem \ref{p2b}; the basic coupled inclusion \eqref{e1.01} of Problem \ref{p1} is recovered by fixing a single parameter value. These hypotheses govern the analysis used in Sections \ref{s5} and \ref{s6}.
	
	\begin{itemize}
		\item[$H(\widetilde{A})$] For each $\zeta \in \Lambda$, the operator $A_\zeta: D(A_\zeta) \subset E \to E$ generates a $C_0$-semigroup $\{T_\zeta(t)\}_{t \geq 0}$ on $E$ such that $T_\zeta(t)$ is compact for every $t > 0$. Moreover, for every bounded set $D \subset \Lambda$, there exists $M_D > 0$ such that
		$$
		\sup_{t \in I} \|T_\zeta(t)\|_{\mathcal{L}(E)} \leq M_D \quad \text{for all } \zeta \in D,
		$$
		and whenever $\zeta_n \to \zeta$ in $\Lambda$,
		$$
		\lim_{n \to \infty} \sup_{t \in I} \|T_{\zeta_n}(t)x - T_\zeta(t)x\|_E = 0 \quad \text{for every } x \in E,
		$$
		while for every $0 < \delta < b$,
		$$
		\lim_{n \to \infty} \sup_{t \in [\delta, b]} \|T_{\zeta_n}(t) - T_\zeta(t)\|_{\mathcal{L}(E)} = 0.
		$$
		
		\item[$H(\widetilde{B})$] For each $(t, u, \eta) \in I \times H \times \Xi$, the operator $B(t, u, \eta): D(B(t, u, \eta)) \subset H \rightrightarrows H$ is maximal monotone, and:
		\begin{enumerate}[{\rm (i)}]
			\item there exist $\theta \in W^{1,2}(I, \mathbb{R})$, $\lambda \in [0, 1)$, and $\lambda' > 0$ such that
			$$
			\operatorname{dis}(B(t, u, \eta_1), B(s, v, \eta_2))
			\leq |\theta(t) - \theta(s)| + \lambda\|u - v\| + \lambda' d_\Xi(\eta_1, \eta_2)
			$$
			for all $t, s \in I$, $u, v \in H$, and $\eta_1, \eta_2 \in \Xi$;
			\item there exists $M_B > 0$ such that
			$$
			\|B^0(t, u, \eta)v\| \leq M_B(1 + \|u\| + \|v\|)
			$$
			for all $t \in I$, $u \in H$, $\eta \in \Xi$, and $v \in D(B(t, u, \eta))$;
			\item for every bounded set $\Omega \subset H \times \Xi$, the set $D(B(I \times \Omega))$ is ball-relatively compact in $H$;
			\item for every bounded set $D \subset \Xi$ and every $r > 0$ there exists $\varrho_D \in L^1(I, \mathbb{R}_+)$ such that
			$$
			\langle v_1 - v_2, u_1 - u_2 \rangle \geq -\varrho_D(t)\|u_1 - u_2\|^2
			$$
			for all $t \in I$, $u_1, u_2 \in r\mathbb{B}_H$, $\eta_1, \eta_2 \in D$, and all $v_i \in B(t, u_i, \eta_i)(u_i)$.
		\end{enumerate}
		
		\item[$H(\widetilde{F})$] The multifunction $F: I \times E \times H \times \Lambda \rightrightarrows E$ has nonempty weakly compact convex values and:
		\begin{enumerate}[{\rm (i)}]
			\item $t \mapsto F(t, x, u, \zeta)$ is measurable for all $(x, u, \zeta) \in E \times H \times \Lambda$;
			\item for every bounded set $D \subset \Lambda$, there exists $k_D \in L^1(I, \mathbb{R}_+)$ such that
			$$
			\mathcal{H}(F(t, x, u, \zeta), F(t, y, v, \zeta))
			\leq k_D(t)(\|x - y\|_E + \|u - v\|)
			$$
			for all $(x, u), (y, v) \in E \times H$, $\zeta \in D$, and a.e. $t \in I$;
			\item for every bounded set $D \subset \Lambda$, there exists $M_F > 0$ such that
			$$
			\|F(t, x, u, \zeta)\|_E \leq M_F(1 + \|x\|_E + \|u\|)
			$$
			for a.e. $t \in I$ and all $(x, u, \zeta) \in E \times H \times D$;
			\item for all $(x, u) \in E \times H$, $\zeta_1, \zeta_2 \in \Lambda$, and a.e. $t \in I$,
			$$
			\mathcal{H}(F(t, x, u, \zeta_1), F(t, x, u, \zeta_2))
			\leq \gamma_1(d_\Lambda(\zeta_1, \zeta_2)) \rho_1(t, \|x\|_E, \|u\|),
			$$
			where $\gamma_1(r) \to 0$ as $r \to 0^+$ and, for every bounded $S \subset L^2(I, E) \times L^2(I, H)$, there exists $\omega_S \in L^1(I, \mathbb{R}_+)$ such that
			$$
			\rho_1(t, \|x(t)\|_E, \|u(t)\|) \leq \omega_S(t)
			$$
			for a.e. $t \in I$ and all $(x, u) \in S$.
		\end{enumerate}
		
		\item[$H(\widetilde{G})$] The multifunction $G: I \times E \times H \times \Xi \rightrightarrows H$ has nonempty weakly compact convex values and:
		\begin{enumerate}[{\rm (i)}]
			\item $t \mapsto G(t, x, u, \eta)$ is measurable;
			\item for every bounded set $D \subset \Xi$, there exists $l_D \in L^1(I, \mathbb{R}_+)$ such that
			$$
			\mathcal{H}(G(t, x, u, \eta), G(t, y, v, \eta))
			\leq l_D(t)(\|x - y\|_E + \|u - v\|)
			$$
			for all $(x, u), (y, v) \in E \times H$, $\eta \in D$, and a.e. $t \in I$;
			\item for every bounded set $D \subset \Xi$, there exists $M_G > 0$ such that
			$$
			\|G(t, x, u, \eta)\| \leq M_G(1 + \|x\|_E + \|u\|)
			$$
			for a.e. $t \in I$ and all $(x, u, \eta) \in E \times H \times D$;
			\item for all $(x, u) \in E \times H$, $\eta_1, \eta_2 \in \Xi$, and a.e. $t \in I$,
			$$
			\mathcal{H}(G(t, x, u, \eta_1), G(t, x, u, \eta_2))
			\leq \gamma_2(d_\Xi(\eta_1, \eta_2)) \rho_2(t, \|x\|_E, \|u\|),
			$$
			where $\gamma_2(r) \to 0$ as $r \to 0^+$ and, for every bounded set $Q \subset C(I, E) \times C(I, H)$, there exists $\omega_Q \in L^1(I, \mathbb{R}_+)$ such that
			$$
			\rho_2(t, \|x(t)\|_E, \|u(t)\|) \leq \omega_Q(t)
			$$
			for a.e. $t \in I$ and all $(x, u) \in Q$.
		\end{enumerate}
		
		\item[$H(0)$] The maps $\zeta \mapsto x_0(\zeta)$ and $\eta \mapsto u_0(\eta)$ are continuous from $\Lambda$ to $E$ and from $\Xi$ to $H$, respectively, and for every bounded set $D \subset \Lambda \times \Xi$ there exists $\varpi_D > 0$ such that
		$$
		\max\{\|x_0(\zeta)\|_E, \|u_0(\eta)\|\} \leq \varpi_D \quad \text{for all } (\zeta, \eta) \in D.
		$$
	\end{itemize}
	
	We now turn to the optimization layer. The next group of hypotheses concerns the Bolza functionals associated with Problem \ref{p2a} and the admissible-control set introduced in Problem \ref{p2b}. Unlike the abstract solvability assumptions above, these conditions are specific to the optimization theory developed in the present paper.
	
	\begin{itemize}
		\item[$H(\varphi)$] The terminal cost function $\varphi: E \times H \times \Lambda \times \Xi \to \mathbb{R}$ is continuous.
		
		\item[$H(\ell)$] The running cost integrand $\ell: I \times E \times H \times \Lambda \times \Xi \to \mathbb{R}$ satisfies the following conditions:
		\begin{enumerate}[{\rm (i)}]
			\item For every $(x, u, \zeta, \eta) \in E \times H \times \Lambda \times \Xi$, the map $t \mapsto \ell(t, x, u, \zeta, \eta)$ is measurable.
			
			\item For a.e. $t \in I$, all $x, y \in E$, and all $(u, \zeta, \eta) \in H \times \Lambda \times \Xi$, the following generalized local continuity condition holds:
			$$
			|\ell(t, x, u, \zeta, \eta) - \ell(t, y, u, \zeta, \eta)| \leq (1 + \|x\|_E \vee \|y\|_E) w(t, \|x - y\|_E),
			$$
			where $\|x\|_E \vee \|y\|_E := \max\{\|x\|_E, \|y\|_E\}$, and $w(t, r)$ is a Carathéodory function on $I \times \mathbb{R}_+$ with values in $\mathbb{R}_+$ such that $w(t, 0) = 0$ for a.e. $t \in I$. Furthermore, for any $R > 0$, there exists a function $\beta_R \in L^1(I, \mathbb{R}_+)$ such that $\sup_{0 \leq r \leq R} w(t, r) \leq \beta_R(t)$ for a.e. $t \in I$.
			
			\item For a.e. $t \in I$ and all $(x, \zeta, \eta) \in E \times \Lambda \times \Xi$, the map $u \mapsto \ell(t, x, u, \zeta, \eta)$ is convex.
			
			\item For any bounded subset $D \subset \Lambda \times \Xi$, there exists a function $\vartheta_D \in L^1(I, \mathbb{R}_+)$ such that for a.e. $t \in I$, all $u \in H$, and all $(\zeta, \eta) \in D$:
			$$
			|\ell(t, 0, u, \zeta, \eta)|\leq \vartheta_D(t)(1 + \|u\|).
			$$
			
			\item For a.e. $t \in I$, the map $(x, u, \zeta, \eta) \mapsto \ell(t, x, u, \zeta, \eta)$ is continuous on $E \times H \times \Lambda \times \Xi$.
		\end{enumerate}
		
		\item[$H(W)$] The multifunction $W: I \rightrightarrows Y$ is measurable, has nonempty, closed, convex values, and is integrably bounded; that is, there exists $\phi_W \in L^2(I, \mathbb{R}_+)$ such that $\sup_{v \in W(t)} \|v\|_Y \leq \phi_W(t)$ a.e.
	\end{itemize}
	
We next turn to the parameterized system in Problem \ref{p2a}, namely \eqref{e1.01a}. Its solution map will serve as the admissible-set map in the optimization sections.
	
	Throughout this part, we work under the standing assumptions $H(\widetilde{A})$, $H(\widetilde{B})$, $H(\widetilde{F})$, $H(\widetilde{G})$, and $H(0)$. Under this convention, the existence, compactness, and continuity properties below are formulated under a single $B$-hypothesis on bounded parameter sets.
	
	For each $(\zeta, \eta) \in \Lambda \times \Xi$, let $\mathcal{S}(\zeta, \eta)$ denote the mild-solution set of Problem \ref{p2a}.
	
\begin{theorem}\label{t4.01}
	Let $D \subset \Lambda \times \Xi$ be bounded. Under hypotheses $H(\widetilde{A})$, $H(\widetilde{B})$, $H(\widetilde{F})$, $H(\widetilde{G})$, and $H(0)$, Problem \ref{p2a} admits at least one mild solution $(x, u) \in C(I, E) \times W^{1,2}(I, H)$ for each $(\zeta, \eta) \in D$.
		
		Moreover, there exist constants $a_D, b_D > 0$ such that every corresponding solution satisfies
		$$
		\|x(t)\|_E + \|u(t)\| \leq a_D \quad \text{for all } t \in I,
		$$
		and
		$$
	\|\dot{u}(t)\| \leq b_D\bigl(1 + |\dot{\theta}(t)|\bigr) \quad \text{for a.e. } t \in I.
	$$
\end{theorem}
	
\begin{theorem}\label{t4.02}
	Let $D \subset \Lambda \times \Xi$ be bounded. Under hypotheses $H(\widetilde{A})$, $H(\widetilde{B})$, $H(\widetilde{F})$, $H(\widetilde{G})$, and $H(0)$, the solution map $\mathcal{S}: \Lambda \times \Xi \rightrightarrows C(I, E) \times C(I, H)$ has nonempty compact values and is Painlev\'e--Kuratowski continuous on $D$. More precisely, if $(\zeta_n, \eta_n) \to (\zeta, \eta)$ in $D$, then
		$$
		(s \times s)\text{-}\limsup_{n \to \infty} \mathcal{S}(\zeta_n, \eta_n) \subseteq \mathcal{S}(\zeta, \eta) \subseteq (s \times s)\text{-}\liminf_{n \to \infty} \mathcal{S}(\zeta_n, \eta_n),
	$$
	in the topology of $C(I, E) \times C(I, H)$.
\end{theorem}
	
	Theorem \ref{t4.02} provides the stability input needed in Sections \ref{s5} and \ref{s6}.
	
	\section{Sensitivity Analysis of Optimal Values and Optimal Pairs}\label{s5}
	
	In this section, we investigate the sensitivity of optimal control problems governed by the parameter-dependent system \eqref{e1.01a}. Building upon the topological properties (specifically, compactness and Painlev\'e--Kuratowski continuity) of the solution set $\mathcal{S}(\zeta, \eta)$ established in Section \ref{s3}, we examine the behavior of the optimal value function and the set of optimal solutions as the parameters $\zeta$ and $\eta$ vary within the metric spaces $\Lambda$ and $\Xi$, respectively.
	
	We consider a general cost functional $\mathcal{J}: C(I, E) \times C(I, H) \times \Lambda \times \Xi \to \mathbb{R}$ defined by:
	\begin{equation}\label{e5.01}
		\mathcal{J}(x, u, \zeta, \eta) := \int_0^b \ell(t, x(t), u(t), \zeta, \eta) \, dt + \varphi(x(b), u(b), \zeta, \eta),
	\end{equation}
	where $\ell: I \times E \times H \times \Lambda \times \Xi \to \mathbb{R}$ represents the running cost integrand, and $\varphi: E \times H \times \Lambda \times \Xi \to \mathbb{R}$ denotes the terminal cost.
	
	For any parameter pair $(\zeta, \eta) \in \Lambda \times \Xi$, the optimal control problem is formulated as the minimization of the cost functional over the set of mild solutions $\mathcal{S}(\zeta, \eta)$.
	
	\begin{problem}\label{p4}
		Find a pair $(x^*, u^*) \in \mathcal{S}(\zeta, \eta)$ such that
		$$
		\mathcal{J}(x^*, u^*, \zeta, \eta) = \inf_{(x, u) \in \mathcal{S}(\zeta, \eta)} \mathcal{J}(x, u, \zeta, \eta).
		$$
	\end{problem}
	
	Associated with Problem \ref{p4}, we define the optimal value function $m: \Lambda \times \Xi \to \mathbb{R}$ by:
	$$
	m(\zeta, \eta) := \inf \{ \mathcal{J}(x, u, \zeta, \eta) \mid (x, u) \in \mathcal{S}(\zeta, \eta) \},
	$$
	and the optimal solution map $\mathcal{M}: \Lambda \times \Xi \rightrightarrows C(I, E) \times C(I, H)$ by:
	\begin{equation}\label{e5.02}
		\mathcal{M}(\zeta, \eta) := \{ (x, u) \in \mathcal{S}(\zeta, \eta) \mid \mathcal{J}(x, u, \zeta, \eta) = m(\zeta, \eta) \}.
	\end{equation}
	
	The cost functional is governed by hypotheses $H(\ell)$ and $H(\varphi)$ from Section \ref{s3}; they encode the measurability, continuity, convexity, and growth conditions used in the direct method below.
	
	Using the compactness of $\mathcal{S}(\zeta,\eta)$ established in Theorem \ref{t4.02}, we show that $\mathcal{J}$ is continuous along every convergent sequence within the feasible set, and apply the direct method of the calculus of variations to establish existence of minimizers.
	
	\begin{proposition}
		Under hypotheses $H(\widetilde{A})$, $H(\widetilde{B})$, $H(\widetilde{F})$, $H(\widetilde{G})$, $H(0)$, $H(\ell)$, and $H(\varphi)$, Problem \ref{p4} admits at least one solution for any $(\zeta, \eta) \in \Lambda \times \Xi$; that is, $\mathcal{M}(\zeta, \eta) \neq \emptyset$.
	\end{proposition}
	
	\begin{proof}
		Fix $(\zeta, \eta) \in \Lambda \times \Xi$ and set $D_0 := \{(\zeta, \eta)\}$. By Theorem \ref{t4.02}, the feasible set $\mathcal{S}(\zeta, \eta)$ is compact in $C(I, E) \times C(I, H)$; in particular, it is nonempty and closed. Hence there exists $R > 0$ such that
		$$
		\|x\|_{C(I,E)} + \|u\|_{C(I,H)} \leq R \quad \text{for all } (x, u) \in \mathcal{S}(\zeta, \eta).
		$$
		Using hypotheses $H(\ell)$(ii) and $H(\ell)$(iv), we obtain for a.e. $t \in I$ and every $(x, u) \in \mathcal{S}(\zeta, \eta)$ that
		\begin{align*}
			|\ell(t, x(t), u(t), \zeta, \eta)|
			&\leq |\ell(t, x(t), u(t), \zeta, \eta) - \ell(t, 0, u(t), \zeta, \eta)| + |\ell(t, 0, u(t), \zeta, \eta)| \\
			&\leq (1 + \|x(t)\|_E)\beta_R(t) + \vartheta_{D_0}(t)(1 + \|u(t)\|) \\
			&\leq (1 + R)\bigl(\beta_R(t) + \vartheta_{D_0}(t)\bigr).
		\end{align*}
		Since $\beta_R, \vartheta_{D_0} \in L^1(I, \mathbb{R}_+)$ and $\varphi$ is continuous, the functional $\mathcal{J}(\cdot, \cdot, \zeta, \eta)$ is finite and bounded from below on the compact set $\mathcal{S}(\zeta, \eta)$. Therefore $m(\zeta, \eta) \in \mathbb{R}$, and we may choose a minimizing sequence $\{(x_n, u_n)\}_{n \in \mathbb{N}} \subseteq \mathcal{S}(\zeta, \eta)$ such that
		$$
		\mathcal{J}(x_n, u_n, \zeta, \eta) \to m(\zeta, \eta).
		$$
		
		By compactness, after passing to a subsequence if necessary, there exists $(x^*, u^*) \in \mathcal{S}(\zeta, \eta)$ such that
		\begin{equation}\label{e5.03}
			x_n \to x^* \text{ strongly in } C(I, E)
			\quad \text{and} \quad
			u_n \to u^* \text{ strongly in } C(I, H).
		\end{equation}
		For a.e. $t \in I$, hypothesis $H(\ell)$(v) implies
		$$
		\ell(t, x_n(t), u_n(t), \zeta, \eta) \to \ell(t, x^*(t), u^*(t), \zeta, \eta).
		$$
		The estimate above supplies an $L^1$-majorant independent of $n$, so the Lebesgue Dominated Convergence Theorem yields
		\begin{equation}\label{e5.04}
			\int_0^b \ell(t, x_n(t), u_n(t), \zeta, \eta) \, dt
			\to
			\int_0^b \ell(t, x^*(t), u^*(t), \zeta, \eta) \, dt.
		\end{equation}
		Moreover, the continuity of $\varphi$ and the uniform convergence in \eqref{e5.03} imply
		\begin{equation}\label{e5.05}
			\varphi(x_n(b), u_n(b), \zeta, \eta)
			\to
			\varphi(x^*(b), u^*(b), \zeta, \eta).
		\end{equation}
		Combining \eqref{e5.04} and \eqref{e5.05}, we obtain
		$$
		\mathcal{J}(x^*, u^*, \zeta, \eta)
		=
		\lim_{n \to \infty} \mathcal{J}(x_n, u_n, \zeta, \eta)
		=
		m(\zeta, \eta).
		$$
		Hence $(x^*, u^*) \in \mathcal{M}(\zeta, \eta)$, and Problem \ref{p4} admits an optimal pair.
	\end{proof}
	
	With the Painlevé-Kuratowski continuity of the solution map established in Section \ref{s3}, we are now in a position to prove the main result of this section: the continuity of the value function.
	
	\begin{theorem}\label{t5.01}
		Under assumptions $H(\widetilde{A})$, $H(\widetilde{B})$, $H(\widetilde{F})$, $H(\widetilde{G})$, $H(0)$, $H(\ell)$, and $H(\varphi)$, the value function $m: \Lambda \times \Xi \to \mathbb{R}$ is continuous.
	\end{theorem}
	
	\begin{proof}
		Let $(\zeta_n, \eta_n) \to (\zeta, \eta)$ in $\Lambda \times \Xi$; we show that $m(\zeta_n,\eta_n) \to m(\zeta,\eta)$. Set $D := \{(\zeta_n, \eta_n) \mid n \in \mathbb{N}\} \cup \{(\zeta, \eta)\}$. Since $(\zeta_n, \eta_n) \to (\zeta, \eta)$, the set $D$ is bounded in $\Lambda \times \Xi$. By Theorem \ref{t4.01}, every pair $(y, v) \in \bigcup_{(\xi, \nu) \in D} \mathcal{S}(\xi, \nu)$ satisfies
		$$
		\|y\|_{C(I, E)} + \|v\|_{C(I, H)} \leq a_D
		$$
		for some constant $a_D > 0$ depending only on $D$.
		
		Set
		$$
		q_D(t) := (1 + a_D)\beta_{a_D}(t) + (1 + a_D)\vartheta_D(t), \quad t \in I,
		$$
		where $\beta_{a_D}$ and $\vartheta_D$ are the functions furnished by hypotheses $H(\ell)$(ii) and $H(\ell)$(iv), respectively. Then, for a.e. $t \in I$, every $(\xi, \nu) \in D$, and every $(y, v) \in \mathcal{S}(\xi, \nu)$, we have
		\begin{align}
			|\ell(t, y(t), v(t), \xi, \nu)|
			&\leq |\ell(t, y(t), v(t), \xi, \nu) - \ell(t, 0, v(t), \xi, \nu)| + |\ell(t, 0, v(t), \xi, \nu)| \notag\\
			&\leq (1 + \|y(t)\|_E)\beta_{a_D}(t) + \vartheta_D(t)(1 + \|v(t)\|) \notag\\
			&\leq q_D(t). \label{e5.06}
		\end{align}
		In particular, $q_D \in L^1(I, \mathbb{R}_+)$ provides a common integrable majorant for all admissible trajectories corresponding to parameters in $D$.
		
		\medskip
		\noindent\textbf{Step 1: Upper semicontinuity.}
		Let $(x, u) \in \mathcal{M}(\zeta, \eta)$ be an optimal pair. By the lower semicontinuity statement in Theorem \ref{t4.02}, there exists a sequence $(x_n, u_n) \in \mathcal{S}(\zeta_n, \eta_n)$ such that
		\begin{equation}\label{e5.07}
		(x_n, u_n) \to (x, u) \quad \text{in } C(I, E) \times C(I, H).
		\end{equation}
		Hypothesis $H(\ell)$(v), the continuity of $\varphi$, the convergence in \eqref{e5.07}, and the common majorant \eqref{e5.06} yield, by the Lebesgue Dominated Convergence Theorem,
		$$
		\mathcal{J}(x_n, u_n, \zeta_n, \eta_n) \to \mathcal{J}(x, u, \zeta, \eta) = m(\zeta, \eta).
		$$
		Since $m(\zeta_n, \eta_n) \leq \mathcal{J}(x_n, u_n, \zeta_n, \eta_n)$, we obtain
		\begin{equation}\label{e5.08}
		\limsup_{n \to \infty} m(\zeta_n, \eta_n) \leq m(\zeta, \eta).
		\end{equation}
		
		\medskip
		\noindent\textbf{Step 2: Lower semicontinuity.}
		Set $\underline{m}:=\liminf_{n\to\infty}m(\zeta_n,\eta_n)$. Choose a subsequence, still denoted by $\{(\zeta_n,\eta_n)\}$, along which the values converge to $\underline{m}$. By the preceding proposition, for each $n$ there exists $(x_n^*,u_n^*)\in\mathcal{M}(\zeta_n,\eta_n)$ such that
		\begin{equation}\label{e5.09}
		\mathcal{J}(x_n^*, u_n^*, \zeta_n, \eta_n)
		= m(\zeta_n, \eta_n) \to \underline{m}.
		\end{equation}
		For each $n$, choose
		$f_n^* \in \mathcal{P}_{F(\cdot,\cdot,\cdot,\zeta_n)}(x_n^*, u_n^*)_I$
		such that
		$$
		x_n^*(t) = T_{\zeta_n}(t)x_0(\zeta_n) + \int_0^t T_{\zeta_n}(t-s)f_n^*(s)\,ds.
		$$
		The uniform estimate in Theorem \ref{t4.01} yields
		$$
		\|x_n^*(t)\|_E + \|u_n^*(t)\| \leq a_D
		\quad \text{for all } t \in I,
		$$
		and
		$$
		\|\dot{u}_n^*(t)\| \leq b_D\bigl(1 + |\dot{\theta}(t)|\bigr)
		\quad \text{for a.e. } t \in I.
		$$
		Hence $\{u_n^*\}$ is equicontinuous on $I$. Since
		$$
		u_n^*(t) \in D\bigl(B(t, u_n^*(t), \eta_n)\bigr)
		\quad \text{and} \quad
		\|u_n^*(t)\| \leq a_D,
		$$
		hypothesis $H(\widetilde{B})$(iii) implies that $\{u_n^*(t)\}_{n \in \mathbb{N}}$ is relatively compact in $H$ for every $t \in I$. The Arzel\`a--Ascoli theorem therefore gives relative compactness of $\{u_n^*\}$ in $C(I, H)$.
		
		Moreover, hypothesis $H(\widetilde{F})$(iii) shows that $\{f_n^*\}$ is bounded in $L^2(I, E)$. Let
		$$
		y_n(t) := T_\zeta(t)x_0(\zeta) + \int_0^t T_\zeta(t-s)f_n^*(s)\,ds.
		$$
		Since the Cauchy operator associated with $T_\zeta$ is compact from $L^2(I, E)$ into $C(I, E)$, the sequence $\{y_n\}$ is relatively compact in $C(I, E)$. On the other hand, by the continuity of $\zeta \mapsto x_0(\zeta)$ and the semigroup continuity in hypothesis $H(\widetilde{A})$, we have
		$$
		\sup_{t \in I}\|T_{\zeta_n}(t)x_0(\zeta_n) - T_\zeta(t)x_0(\zeta)\|_E \to 0.
		$$
		Furthermore, if
		$$
		\varepsilon_{n,\delta} := \sup_{r \in [\delta, b]}\|T_{\zeta_n}(r) - T_\zeta(r)\|_{\mathcal{L}(E)},
		$$
		then hypothesis $H(\widetilde{A})$ implies $\varepsilon_{n,\delta} \to 0$ for every fixed $\delta > 0$, while the boundedness of $\{f_n^*\}$ in $L^2(I, E)$ gives
		\begin{align*}
			&\sup_{t \in I}\int_0^t \|[T_{\zeta_n}(t-s) - T_\zeta(t-s)]f_n^*(s)\|_E\,ds \\
			&\leq \varepsilon_{n,\delta} b^{1/2}\sup_n\|f_n^*\|_{L^2(I, E)} + 2M_D\delta^{1/2}\sup_n\|f_n^*\|_{L^2(I, E)}.
		\end{align*}
		Passing first to the limit as $n \to \infty$ and then letting $\delta \downarrow 0$, we infer that
		$$
		\|x_n^* - y_n\|_{C(I, E)} \to 0.
		$$
		Therefore $\{x_n^*\}$ is relatively compact in $C(I, E)$, and so $\{(x_n^*, u_n^*)\}$ is relatively compact in $C(I, E) \times C(I, H)$. Passing to a further subsequence if needed and using the upper semicontinuity statement in Theorem \ref{t4.02}, we may assume that
		\begin{equation}\label{e5.10}
			(x_n^*, u_n^*) \to (\bar{x}, \bar{u}) \quad \text{in } C(I, E) \times C(I, H),
			\qquad
			(\bar{x}, \bar{u}) \in \mathcal{S}(\zeta, \eta).
		\end{equation}
		Hypothesis $H(\ell)$(v), the continuity of $\varphi$, \eqref{e5.06}, and \eqref{e5.10} yield, again by the Lebesgue Dominated Convergence Theorem,
		$$
		\mathcal{J}(x_n^*, u_n^*, \zeta_n, \eta_n) \to \mathcal{J}(\bar{x}, \bar{u}, \zeta, \eta).
		$$
		Consequently, \eqref{e5.09} gives
		\begin{equation}\label{e5.11}
		\liminf_{n \to \infty} m(\zeta_n, \eta_n)
		=
		\underline{m}
		=
		\mathcal{J}(\bar{x}, \bar{u}, \zeta, \eta)
		\geq
		m(\zeta, \eta),
		\end{equation}
		where the inequality follows from the feasibility in \eqref{e5.10}.
		
		Combining \eqref{e5.08} and \eqref{e5.11} proves that $m(\zeta_n,\eta_n) \to m(\zeta,\eta)$.
	\end{proof}

	\begin{theorem}\label{t5.02}
		Under the assumptions of Theorem \ref{t5.01}, the multivalued map of optimal pairs $\mathcal{M}: \Lambda \times \Xi \rightrightarrows C(I, E) \times C(I, H)$ defined by \eqref{e5.02} is upper semicontinuous.
	\end{theorem}
	
	\begin{proof}
		By Proposition \ref{p2.01}, $\mathcal{M}$ is u.s.c. if for any closed subset $\mathcal{D} \subseteq C(I, E) \times C(I, H)$, the set
		\begin{equation}\label{e5.12}
		\mathcal{M}^{-}(\mathcal{D}) := \left\{ (\zeta, \eta) \in \Lambda \times \Xi \mid \mathcal{M}(\zeta, \eta) \cap \mathcal{D} \neq \emptyset \right\}
		\end{equation}
		is closed in $\Lambda \times \Xi$.
		
		Let $\{(\zeta_n, \eta_n)\} \subseteq \mathcal{M}^{-}(\mathcal{D})$ converge to $(\zeta, \eta)$. We show that $(\zeta, \eta) \in \mathcal{M}^{-}(\mathcal{D})$.
		For each $n$, choose $(x_n, u_n) \in \mathcal{M}(\zeta_n, \eta_n) \cap \mathcal{D}$. Then
		\begin{equation}\label{e5.13}
		(x_n, u_n) \in \mathcal{S}(\zeta_n, \eta_n)\cap\mathcal{D}
		\quad \text{and} \quad
		\mathcal{J}(x_n, u_n, \zeta_n, \eta_n) = m(\zeta_n, \eta_n).
		\end{equation}
		
		Set
		$$
		D_0 := \{(\zeta_n, \eta_n) \mid n \in \mathbb{N}\} \cup \{(\zeta, \eta)\}.
		$$
		Since $(\zeta_n, \eta_n) \to (\zeta, \eta)$, the set $D_0$ is bounded in $\Lambda \times \Xi$. For each $n$, choose
		$f_n \in \mathcal{P}_{F(\cdot,\cdot,\cdot,\zeta_n)}(x_n, u_n)_I$
		such that
		$$
		x_n(t) = T_{\zeta_n}(t)x_0(\zeta_n) + \int_0^t T_{\zeta_n}(t-s)f_n(s)\,ds.
		$$
		The uniform estimate in Theorem \ref{t4.01} yields
		$$
		\|x_n(t)\|_E + \|u_n(t)\| \leq a_{D_0}
		\quad \text{for all } t \in I,
		$$
		and
		$$
		\|\dot{u}_n(t)\| \leq b_{D_0}\bigl(1 + |\dot{\theta}(t)|\bigr)
		\quad \text{for a.e. } t \in I.
		$$
		Thus $\{u_n\}$ is equicontinuous on $I$. Since
		$$
		u_n(t) \in D\bigl(B(t, u_n(t), \eta_n)\bigr)
		\quad \text{and} \quad
		\|u_n(t)\| \leq a_{D_0},
		$$
		hypothesis $H(\widetilde{B})$(iii) implies that $\{u_n(t)\}_{n \in \mathbb{N}}$ is relatively compact in $H$ for every $t \in I$, and therefore $\{u_n\}$ is relatively compact in $C(I, H)$.
		
		Hypothesis $H(\widetilde{F})$(iii) shows that $\{f_n\}$ is bounded in $L^2(I, E)$. Define
		$$
		y_n(t) := T_\zeta(t)x_0(\zeta) + \int_0^t T_\zeta(t-s)f_n(s)\,ds.
		$$
		As above, the compactness of the Cauchy operator associated with $T_\zeta$ implies that $\{y_n\}$ is relatively compact in $C(I, E)$. Using the continuity of the initial data map and the semigroup continuity from $H(\widetilde{A})$, together with the estimate
		\begin{align*}
			&\sup_{t \in I}\int_0^t \|[T_{\zeta_n}(t-s) - T_\zeta(t-s)]f_n(s)\|_E\,ds \\
			&\leq \varepsilon_{n,\delta} b^{1/2}\sup_n\|f_n\|_{L^2(I, E)} + 2M_{D_0}\delta^{1/2}\sup_n\|f_n\|_{L^2(I, E)},
		\end{align*}
		where $\varepsilon_{n,\delta} := \sup_{r \in [\delta, b]}\|T_{\zeta_n}(r) - T_\zeta(r)\|_{\mathcal{L}(E)}$, we obtain
		$$
		\|x_n - y_n\|_{C(I, E)} \to 0.
		$$
		Hence $\{x_n\}$ is relatively compact in $C(I, E)$, and so $\{(x_n, u_n)\}$ is relatively compact in $C(I, E) \times C(I, H)$. Passing to a subsequence and using the upper semicontinuity statement in Theorem \ref{t4.02}, we may assume that
		\begin{equation}\label{e5.14}
			(x_n,u_n)\to(x,u)\quad\text{in }C(I,E)\times C(I,H),
			\qquad
			(x,u)\in\mathcal{S}(\zeta,\eta).
		\end{equation}
		Set
		$$
		q_{D_0}(t) := (1 + a_{D_0})\beta_{a_{D_0}}(t) + (1 + a_{D_0})\vartheta_{D_0}(t), \quad t \in I.
		$$
		Then $q_{D_0}\in L^1(I,\mathbb{R}_+)$ and
		\begin{equation}\label{e5.15}
		|\ell(t, x_n(t), u_n(t), \zeta_n, \eta_n)| \leq q_{D_0}(t)
		\quad \text{for a.e. } t \in I \text{ and all } n,
		\end{equation}
		Hypothesis $H(\ell)$(v), the continuity of $\varphi$, \eqref{e5.14}, and \eqref{e5.15} yield, by the Lebesgue Dominated Convergence Theorem,
		$$
		\mathcal{J}(x_n, u_n, \zeta_n, \eta_n) \to \mathcal{J}(x, u, \zeta, \eta).
		$$
		Combining this convergence with \eqref{e5.13} and the continuity of $m$ from Theorem \ref{t5.01}, we obtain
		$$
		\mathcal{J}(x, u, \zeta, \eta)
		=
		\lim_{n \to \infty} \mathcal{J}(x_n, u_n, \zeta_n, \eta_n)
		=
		\lim_{n \to \infty} m(\zeta_n, \eta_n)
		=
		m(\zeta, \eta).
		$$
		Thus $(x, u) \in \mathcal{M}(\zeta, \eta)$. Moreover, \eqref{e5.13}, \eqref{e5.14}, and the closedness of $\mathcal{D}$ imply $(x,u)\in\mathcal{D}$.
		
		Consequently, $(x, u) \in \mathcal{M}(\zeta, \eta) \cap \mathcal{D}$, so \eqref{e5.12} gives $(\zeta, \eta) \in \mathcal{M}^{-}(\mathcal{D})$. Thus, $\mathcal{M}^{-}(\mathcal{D})$ is closed, and Proposition \ref{p2.01} shows that $\mathcal{M}$ is upper semicontinuous.
	\end{proof}
	
\section{Optimal Control, Joint Parameter Design, and Compromise Robustness}\label{s6}
	
	In this section, building on the sensitivity analysis from Section \ref{s5}, we study a family of optimal control models for the coupled evolution inclusion. We begin with the basic problem in which the parameters $(\zeta, \eta)$ are fixed and the optimization variable is the external control $w$. We then extend the analysis to simultaneous parameter design and control, and finally to robust and Hurwicz-type compromise formulations under parameter uncertainty.
	
	We now return to the control-driven system \eqref{e1.03} introduced in Problem \ref{p2b}. Throughout this section, $Y$ denotes the control space and $\mathcal{W}_{ad}$ is the admissible-control set defined in the introduction. In the first part of this section, the parameters $(\zeta, \eta) \in \Lambda \times \Xi$ are regarded as fixed.
	
	With the admissible control set in place, the optimal control problem for a fixed parameter pair is formulated as follows.
	\begin{problem}\label{p5}
		For fixed parameters $(\zeta, \eta) \in \Lambda \times \Xi$, find an optimal control $w^* \in \mathcal{W}_{ad}$ such that
		$$
		\mathcal{F}_{\zeta,\eta}(w^*) = \inf_{w \in \mathcal{W}_{ad}} \mathcal{F}_{\zeta,\eta}(w),
		$$
		where
		$$
		\mathcal{F}_{\zeta,\eta}(w) := m_{\zeta,\eta}(w) + \mathcal{R}(w).
		$$
	\end{problem}
	Here, $m_{\zeta,\eta}(w)$ denotes the optimal value of the state-constrained problem associated with the control $w$ and the fixed parameters $(\zeta, \eta)$:
	$$
	m_{\zeta,\eta}(w) := \inf_{(x, u) \in \mathcal{S}_{\zeta,\eta}(w)} \mathcal{J}_{\zeta,\eta}(x, u),
	$$
	and the associated Bolza functional is
	$$
	\mathcal{J}_{\zeta,\eta}(x, u) := \int_0^b \ell(t, x(t), u(t), \zeta, \eta)\,dt + \varphi(x(b), u(b), \zeta, \eta),
	$$
	and the regularization term is defined by
	\begin{equation}\label{e6.02}
		\mathcal{R}(w) := \frac{1}{2} \int_0^b \|w(t)\|_Y^2 \, dt.
	\end{equation}
	
	To establish the existence of optimal controls, we use the direct method of the calculus of variations. The main point is to combine weak compactness in the control space with strong compactness of the corresponding state trajectories.
	
	Hypothesis $H(W)$ from Section \ref{s3} ensures that $\mathcal{W}_{ad}$ is nonempty, bounded, closed, and convex in $L^2(I, Y)$.
	
	The next lemma is the key compactness tool: weak convergence of controls in $L^2(I, Y)$ yields strong compactness of the associated state trajectories in $C(I, E) \times C(I, H)$. This relies on the compactness of the semigroup generated by $A_\zeta$ from hypothesis $H(\widetilde{A})$.
	
	\begin{lemma}\label{l6.01}
		Fix $(\zeta, \eta) \in \Lambda \times \Xi$ and assume that hypotheses $H(\widetilde{A})$, $H(\widetilde{B})$, $H(\widetilde{F})$, $H(\widetilde{G})$, and $H(W)$ are satisfied. Let $\{w_n\} \subset \mathcal{W}_{ad}$ be a sequence converging weakly to $w$ in $L^2(I, Y)$. Let $(x_n, u_n) \in \mathcal{S}_{\zeta,\eta}(w_n)$ be a sequence of solutions corresponding to $w_n$. Then $w \in \mathcal{W}_{ad}$, the sequence $\{(x_n, u_n)\}_n$ is relatively compact in $C(I, E) \times C(I, H)$, and any limit point $(x, u)$ belongs to $\mathcal{S}_{\zeta,\eta}(w)$.
	\end{lemma}

	\begin{proof}
		Since $\mathcal{W}_{ad}$ is closed and convex in $L^2(I, Y)$, the weak convergence $w_n \rightharpoonup w$ implies $w \in \mathcal{W}_{ad}$. Moreover, hypothesis $H(W)$ yields
		$$
		\|w_n(t)\|_Y \leq \phi_W(t) \quad \text{for a.e. } t \in I \text{ and all } n \in \mathbb{N},
		$$
		so that $\{\mathcal{B}w_n\}$ is bounded in $L^2(I, E)$.
		
		Arguing as in the a priori estimates behind Theorem \ref{t4.01}, with the term $\mathcal{B}w_n$ treated as an additional source in the $x$-equation, we obtain constants $a, b > 0$, independent of $n$, such that
		\begin{equation}\label{e6.03}
			\|x_n(t)\|_E+\|u_n(t)\|\leq a\quad(t\in I),
			\qquad
			\|\dot u_n(t)\|\leq b\bigl(1+|\dot\theta(t)|\bigr)\quad\text{a.e. }t\in I.
		\end{equation}
		In particular, \eqref{e6.03} implies that $\{u_n\}$ is equicontinuous on $I$. Since
		$$
		u_n(t) \in D\bigl(B(t, u_n(t), \eta)\bigr)
		\quad \text{and} \quad
		\|u_n(t)\| \leq a,
		$$
		hypothesis $H(\widetilde{B})$(iii) implies that $\{u_n(t)\}_{n \in \mathbb{N}}$ is relatively compact in $H$ for every $t \in I$. Hence, by the Arzel\`a--Ascoli theorem, $\{u_n\}$ is relatively compact in $C(I, H)$.
		
		For each $n$, choose selections
		$$
		f_n \in \mathcal{P}_{F(\cdot, \cdot, \cdot, \zeta)}(x_n, u_n)_I
		\quad \text{and} \quad
		g_n \in L^2(I, H)
		$$
		such that $g_n(t) \in G(t, x_n(t), u_n(t), \eta)$ a.e. and
		\begin{equation}\label{e6.04}
		x_n(t) = T_\zeta(t)x_0(\zeta) + \int_0^t T_\zeta(t-s) f_n(s)\,ds + \int_0^t T_\zeta(t-s)\mathcal{B}w_n(s)\,ds.
		\end{equation}
		By the linear growth assumptions in $H(\widetilde{F})$ and $H(\widetilde{G})$, together with \eqref{e6.03}, the sequences $\{f_n\}$ and $\{g_n\}$ are bounded in $L^2(I, E)$ and $L^2(I, H)$, respectively.
		
		Let $\mathcal{G}_\zeta: L^2(I, E) \to C(I, E)$ be the Cauchy operator
		\begin{equation}\label{e6.05}
		(\mathcal{G}_\zeta f)(t) := \int_0^t T_\zeta(t-s)f(s)\,ds.
		\end{equation}
		We claim that $\mathcal{G}_\zeta$ is compact. Let $\{h_n\}$ be a bounded sequence in $L^2(I, E)$, and set $z_n := \mathcal{G}_\zeta h_n$. Then $\{z_n\}$ is uniformly bounded in $C(I, E)$. Moreover, for $0 \leq s < t \leq b$,
		\begin{align*}
			\|z_n(t) - z_n(s)\|_E
			&\leq \left\|(T_\zeta(t-s)-I)\int_0^s T_\zeta(s-r)h_n(r)\,dr\right\|_E + \int_s^t \|T_\zeta(t-r)h_n(r)\|_E\,dr \\
			&\leq \left\|(T_\zeta(t-s)-I)\int_0^s T_\zeta(s-r)h_n(r)\,dr\right\|_E + M_\zeta |t-s|^{1/2}\|h_n\|_{L^2(I, E)},
		\end{align*}
		so the strong continuity of the semigroup implies equicontinuity. For each fixed $t \in (0, b]$ and $\delta \in (0, t)$, we split
		$$
		z_n(t)
		=
		T_\zeta(\delta)\int_0^{t-\delta} T_\zeta(t-\delta-s)h_n(s)\,ds
		+
		\int_{t-\delta}^t T_\zeta(t-s)h_n(s)\,ds.
		$$
		The first term is relatively compact because $T_\zeta(\delta)$ is compact and the inner integrals form a bounded set in $E$, while the second term is uniformly small in $n$ when $\delta \downarrow 0$. Hence $\{z_n(t)\}$ is relatively compact in $E$ for every $t \in I$, and the Arzel\`a--Ascoli theorem shows that $\mathcal{G}_\zeta$ is compact. Define also
		$$
		(\widetilde{\mathcal{B}}v)(t) := \mathcal{B}v(t), \qquad v \in L^2(I, Y).
		$$
		Then $\widetilde{\mathcal{B}}:L^2(I,Y)\to L^2(I,E)$ is bounded, and compactness of $\mathcal{G}_\zeta$ gives
		\begin{equation}\label{e6.06}
			\mathcal{L}_\zeta:=\mathcal{G}_\zeta\circ\widetilde{\mathcal{B}}
			:L^2(I,Y)\to C(I,E)\ \text{compact},
			\qquad
			\mathcal{L}_\zeta w_n\to\mathcal{L}_\zeta w\ \text{in }C(I,E).
		\end{equation}
		Because $\{f_n\}$ is bounded and $\mathcal{G}_\zeta$ is compact, \eqref{e6.04}--\eqref{e6.06} imply that $\{x_n\}$ is relatively compact in $C(I,E)$. Therefore $\{(x_n,u_n)\}$ is relatively compact in $C(I,E)\times C(I,H)$.
		
		Let $(x,u)$ be a limit point. After passing to a subsequence, we may assume that
		\begin{equation}\label{e6.07}
			\begin{aligned}
				(x_n,u_n)&\to(x,u) &&\text{in }C(I,E)\times C(I,H),\\
				f_n&\rightharpoonup f &&\text{in }L^2(I,E),\\
				g_n&\rightharpoonup g,\qquad \dot u_n\rightharpoonup\dot u
				&&\text{in }L^2(I,H).
			\end{aligned}
		\end{equation}
		Passing to the limit in \eqref{e6.04} by \eqref{e6.05}--\eqref{e6.07} gives
		\begin{equation}\label{e6.08}
		x(t) = T_\zeta(t)x_0(\zeta) + \int_0^t T_\zeta(t-s)f(s)\,ds + \int_0^t T_\zeta(t-s)\mathcal{B}w(s)\,ds.
		\end{equation}
		By the convergence theorem for Aumann integrals and \eqref{e6.07}, $f(t)$ and $g(t)$ belong to the closed convex hulls of the corresponding weak outer limits for a.e. $t\in I$. Hypotheses $H(\widetilde{F})$(ii) and $H(\widetilde{G})$(ii), together with the uniform convergence in \eqref{e6.07}, imply that the Hausdorff distances from the approximating values to the limiting values tend to zero a.e. Since the limiting values are closed and convex, the weak outer limits are contained in them. Consequently,
		\begin{equation}\label{e6.09}
		f(t) \in F(t, x(t), u(t), \zeta)
		\quad \text{and} \quad
		g(t) \in G(t, x(t), u(t), \eta)
		\quad \text{for a.e. } t \in I.
		\end{equation}
		Finally, set
		\begin{equation}\label{e6.10}
			\omega_n:=-\dot u_n+g_n\rightharpoonup
			\omega:=-\dot u+g\quad\text{in }L^2(I,H).
		\end{equation}
		Mazur's theorem yields convex combinations of $\{\omega_n\}$ converging strongly to $\omega$ in $L^2(I,H)$ and, along a subsequence, pointwise a.e. Moreover, \eqref{e6.03} and the growth bound on $G$ give $\|\omega_n(t)\|\leq b(1+|\dot\theta(t)|)+M_G(1+a)$ a.e. Thus every weak cluster point of $\{\omega_n(t)\}$ coincides with the Mazur limit, and $\omega_n(t)\rightharpoonup\omega(t)$ in $H$ for a.e. $t\in I$. Fix such a $t$. Since
		$$
		\operatorname{dis}\bigl(B(t, u_n(t), \eta), B(t, u(t), \eta)\bigr)
		\leq
		\lambda\|u_n(t)-u(t)\| \to 0,
		$$
		Lemma \ref{l2.02}(i) yields $u(t) \in D(B(t, u(t), \eta))$. Let $\xi \in D(B(t, u(t), \eta))$. By Lemma \ref{l2.02}(ii), there exists $\xi_n \in D(B(t, u_n(t), \eta))$ such that
		$$
		\xi_n \to \xi
		\quad \text{and} \quad
		B^0(t, u_n(t), \eta)\xi_n \to B^0(t, u(t), \eta)\xi
		\quad \text{in } H.
		$$
		Since $\omega_n(t) \in B(t, u_n(t), \eta)(u_n(t))$, monotonicity gives
		$$
		\langle \omega_n(t), u_n(t)-\xi_n \rangle
		\geq
		\langle B^0(t, u_n(t), \eta)\xi_n, u_n(t)-\xi_n \rangle.
		$$
		Using the pointwise weak convergence following from \eqref{e6.10}, the strong convergence in \eqref{e6.07}, and $\xi_n\to\xi$, we pass to the limit to obtain
		$$
		\langle \omega(t), u(t)-\xi \rangle
		\geq
		\langle B^0(t, u(t), \eta)\xi, u(t)-\xi \rangle.
		$$
		Lemma \ref{l2.02}(iii) therefore implies
		\begin{equation}\label{e6.11}
		-\dot{u}(t) + g(t) = \omega(t) \in B(t, u(t), \eta)(u(t)) \quad \text{for a.e. } t \in I.
		\end{equation}
		Equations \eqref{e6.08}, \eqref{e6.09}, and \eqref{e6.11} show that $(x,u)$ is a mild solution of \eqref{e1.03} corresponding to $w$; hence
		$$
		(x, u) \in \mathcal{S}_{\zeta,\eta}(w).
		$$
	\end{proof}
	
	We now establish the lower semicontinuity of the inner value function.
	
	\begin{proposition}\label{p6.01}
		Assume that the hypotheses of Lemma \ref{l6.01} are satisfied. If, in addition, hypotheses $H(\ell)$ and $H(\varphi)$ hold, then the value function $m_{\zeta,\eta}: \mathcal{W}_{ad} \to \mathbb{R}$ is sequentially lower semicontinuous with respect to the weak topology of $L^2(I, Y)$. Specifically,
		$$
		w_n \rightharpoonup w \text{ in } L^2(I, Y) \implies m_{\zeta,\eta}(w) \leq \liminf_{n \to \infty} m_{\zeta,\eta}(w_n).
		$$
	\end{proposition}
	
	\begin{proof}
		Let $w_n\rightharpoonup w$ in $L^2(I,Y)$ and set $\underline m:=\liminf_{n\to\infty}m_{\zeta,\eta}(w_n)$. Passing to a subsequence without relabeling and choosing $(x_n,u_n)\in\mathcal{S}_{\zeta,\eta}(w_n)$ by the definition of the infimum, we may arrange that
		\begin{equation}\label{e6.12}
			\begin{aligned}
				m_{\zeta,\eta}(w_n)
				&\leq\mathcal{J}_{\zeta,\eta}(x_n,u_n)
				\leq m_{\zeta,\eta}(w_n)+\frac1n,\\
				m_{\zeta,\eta}(w_n)&\to\underline m.
			\end{aligned}
		\end{equation}
		By Lemma \ref{l6.01}, after passing to a further subsequence if necessary, we may assume that
		\begin{equation}\label{e6.13}
			(x_n,u_n)\to(x,u)\quad\text{in }C(I,E)\times C(I,H),
			\qquad
			(x,u)\in\mathcal{S}_{\zeta,\eta}(w).
		\end{equation}
		
		Choose $R > 0$ such that
		$$
		\|x_n\|_{C(I,E)} + \|u_n\|_{C(I,H)} + \|x\|_{C(I,E)} + \|u\|_{C(I,H)} \leq R
		\quad \text{for all } n \in \mathbb{N}.
		$$
		Set
		$$
		q_{\zeta,\eta}(t) := (1 + R)\beta_R(t) + (1 + R)\vartheta_{\{(\zeta,\eta)\}}(t), \quad t \in I.
		$$
		Then $q_{\zeta,\eta}\in L^1(I,\mathbb{R}_+)$, and hypotheses $H(\ell)$(ii) and $H(\ell)$(iv) imply
		\begin{equation}\label{e6.14}
		|\ell(t, x_n(t), u_n(t), \zeta, \eta)| \leq q_{\zeta,\eta}(t)
		\quad \text{for a.e. } t \in I \text{ and all } n.
		\end{equation}
		Hypothesis $H(\ell)$(v), the continuity of $\varphi$, \eqref{e6.13}, and \eqref{e6.14} yield, by the Lebesgue Dominated Convergence Theorem,
		\begin{equation}\label{e6.15}
		\mathcal{J}_{\zeta,\eta}(x_n, u_n) \to \mathcal{J}_{\zeta,\eta}(x, u).
		\end{equation}
		Combining \eqref{e6.12}, \eqref{e6.13}, and \eqref{e6.15} gives
		$$
		m_{\zeta,\eta}(w)
		\leq\mathcal{J}_{\zeta,\eta}(x,u)
		=\underline m
		=\liminf_{n\to\infty}m_{\zeta,\eta}(w_n).
		$$
	\end{proof}
	
	With these preparations, we are now in a position to state the main existence theorem for the optimal control problem under a fixed parameter configuration.
	
	\begin{theorem}\label{t6.01}
		Suppose that the hypotheses of Proposition \ref{p6.01} hold. Then, Problem \ref{p5} admits at least one optimal control $w^* \in \mathcal{W}_{ad}$.
	\end{theorem}
	
	\begin{proof}
		Let $\{w_k\}\subset\mathcal{W}_{ad}$ be a minimizing sequence for $\mathcal{F}_{\zeta,\eta}$. Hypothesis $H(W)$ and reflexivity of $L^2(I,Y)$ imply that $\mathcal{W}_{ad}$ is sequentially weakly compact. Hence, after passing to a subsequence,
		\begin{equation}\label{e6.16}
			\mathcal{F}_{\zeta,\eta}(w_k)\to
			\inf_{w\in\mathcal{W}_{ad}}\mathcal{F}_{\zeta,\eta}(w),
			\qquad
			w_k\rightharpoonup w^*\in\mathcal{W}_{ad}.
		\end{equation}
		By \eqref{e6.02}, $\mathcal{R}$ is sequentially weakly lower semicontinuous. Proposition \ref{p6.01} gives the same property for $m_{\zeta,\eta}$, and therefore
		\begin{equation}\label{e6.17}
			\mathcal{F}_{\zeta,\eta}(w^*)
			\leq\liminf_{k\to\infty}\mathcal{F}_{\zeta,\eta}(w_k).
		\end{equation}
		Equations \eqref{e6.16}--\eqref{e6.17} yield
		$\mathcal{F}_{\zeta,\eta}(w^*)\leq\inf_{w\in\mathcal{W}_{ad}}\mathcal{F}_{\zeta,\eta}(w)$.
		The reverse inequality follows from $w^*\in\mathcal{W}_{ad}$, so equality holds and $w^*$ solves Problem \ref{p5}.
	\end{proof}
	
	We next consider a joint design problem in which the control and the structural parameters are optimized simultaneously. This extension is natural in applications where one seeks both a suitable input signal and a favorable parameter configuration minimizing the total performance index.
	
	Let $\mathcal{P} \subset \Lambda \times \Xi$ be a nonempty compact set representing the admissible design parameter space. The simultaneous optimal design and control problem is formulated as follows:
	
	\begin{problem}\label{p6}
		Find a control-parameter tuple $(w^*, \zeta^*, \eta^*) \in \mathcal{W}_{ad} \times \mathcal{P}$ such that:
		$$
		\mathcal{F}_{joint}(w^*, \zeta^*, \eta^*) = \inf_{(w, \zeta, \eta) \in \mathcal{W}_{ad} \times \mathcal{P}} \left\{ m_{\zeta, \eta}(w) + \mathcal{R}(w) \right\},
		$$
		where $m_{\zeta, \eta}(w)$ is the optimal value of the inner problem corresponding to the parameters $(\zeta, \eta)$ and control $w$.
	\end{problem}
	Equivalently,
	$$
	\mathcal{F}_{joint}(w, \zeta, \eta) := m_{\zeta, \eta}(w) + \mathcal{R}(w)
	\qquad
	\text{for } (w, \zeta, \eta) \in \mathcal{W}_{ad} \times \mathcal{P}.
	$$
	
	To establish solvability of Problem \ref{p6}, we combine weak compactness in the control space with compactness of the associated state trajectories under simultaneous perturbations of the control and the parameters.
	
	\begin{theorem}\label{t6.02}
		Suppose that the hypotheses of Proposition \ref{p6.01} hold, and assume in addition hypothesis $H(0)$. If the parameter set $\mathcal{P} \subset \Lambda \times \Xi$ is compact, then Problem \ref{p6} admits at least one optimal triplet $(w^*, \zeta^*, \eta^*) \in \mathcal{W}_{ad} \times \mathcal{P}$.
	\end{theorem}
	
	\begin{proof}
		Let $\{(w_k,\zeta_k,\eta_k)\}\subset\mathcal{W}_{ad}\times\mathcal P$ be a minimizing sequence. Compactness of $\mathcal P$, hypothesis $H(W)$, and reflexivity of $L^2(I,Y)$ allow us to pass to a subsequence such that
		\begin{equation}\label{e6.18}
			\begin{aligned}
				\mathcal F_{\rm joint}(w_k,\zeta_k,\eta_k)
				&\to\inf_{\mathcal{W}_{ad}\times\mathcal P}\mathcal F_{\rm joint},\\
				(\zeta_k,\eta_k)&\to(\zeta^*,\eta^*)\in\mathcal P,\\
				w_k&\rightharpoonup w^*\in\mathcal{W}_{ad}
				\quad\text{in }L^2(I,Y).
			\end{aligned}
		\end{equation}
		For each $k$, choose $(x_k,u_k)$ by the definition of the inner infimum so that
		\begin{equation}\label{e6.19}
			(x_k,u_k)\in\mathcal S_{\zeta_k,\eta_k}(w_k),
			\qquad
			\mathcal J_{\zeta_k,\eta_k}(x_k,u_k)
			\leq m_{\zeta_k,\eta_k}(w_k)+\frac1k.
		\end{equation}
		
		Set $D_0:=\{(\zeta_k,\eta_k)\mid k\in\mathbb N\}\cup\{(\zeta^*,\eta^*)\}$. The control bound in $H(W)$ and the a priori estimates used for Theorem \ref{t4.01} give constants $a_0,b_0>0$, independent of $k$, such that
		\begin{equation}\label{e6.20}
			\|x_k(t)\|_E+\|u_k(t)\|\leq a_0\quad(t\in I),
			\qquad
			\|\dot u_k(t)\|\leq b_0\bigl(1+|\dot\theta(t)|\bigr)
			\quad\text{a.e. }t\in I.
		\end{equation}
		Thus $\{u_k\}$ is equicontinuous. Since
		$u_k(t)\in D(B(t,u_k(t),\eta_k))$, hypothesis $H(\widetilde B)$(iii), \eqref{e6.20}, and the Arzel\`a--Ascoli theorem show that $\{u_k\}$ is relatively compact in $C(I,H)$.
		
		Choose $f_k\in\mathcal P_{F(\cdot,\cdot,\cdot,\zeta_k)}(x_k,u_k)_I$ and
		$g_k\in L^2(I,H)$ with $g_k(t)\in G(t,x_k(t),u_k(t),\eta_k)$ a.e. Put
		$h_k:=f_k+\mathcal B w_k$. By the growth assumptions and \eqref{e6.20}, the sequences $\{f_k\}$, $\{g_k\}$, and $\{h_k\}$ are bounded in their respective $L^2$ spaces, and
		$$
		x_k(t)=T_{\zeta_k}(t)x_0(\zeta_k)
		+\int_0^tT_{\zeta_k}(t-s)h_k(s)\,ds.
		$$
		Define
		$$
		y_k(t):=T_{\zeta^*}(t)x_0(\zeta^*)
		+\int_0^tT_{\zeta^*}(t-s)h_k(s)\,ds.
		$$
		By the compactness of the Cauchy operator in \eqref{e6.05}, $\{y_k\}$ is relatively compact in $C(I,E)$. For $\delta\in(0,b)$, set
		$\varepsilon_{k,\delta}:=\sup_{r\in[\delta,b]}
		\|T_{\zeta_k}(r)-T_{\zeta^*}(r)\|_{\mathcal L(E)}$.
		The standard semigroup estimate gives
		\begin{align*}
			\|x_k-y_k\|_{C(I,E)}
			&\leq\sup_{t\in I}
			\|T_{\zeta_k}(t)x_0(\zeta_k)-T_{\zeta^*}(t)x_0(\zeta^*)\|_E\\
			&\quad+\varepsilon_{k,\delta}b^{1/2}\sup_k\|h_k\|_{L^2(I,E)}
			+2M_{D_0}\delta^{1/2}\sup_k\|h_k\|_{L^2(I,E)}.
		\end{align*}
		Hypotheses $H(\widetilde A)$ and $H(0)$ imply, first as $k\to\infty$ and then as $\delta\downarrow0$, that
		$\|x_k-y_k\|_{C(I,E)}\to0$. Hence $\{(x_k,u_k)\}$ is relatively compact. After passing to a further subsequence,
		\begin{equation}\label{e6.21}
			\begin{aligned}
				(x_k,u_k)&\to(x^*,u^*) &&\text{in }C(I,E)\times C(I,H),\\
				f_k&\rightharpoonup f &&\text{in }L^2(I,E),\\
				g_k&\rightharpoonup g,\qquad \dot u_k\rightharpoonup\dot u^*
				&&\text{in }L^2(I,H).
			\end{aligned}
		\end{equation}
		
		Using \eqref{e6.05}--\eqref{e6.06}, \eqref{e6.18}, and \eqref{e6.21}, we may pass to the limit in the mild equation and obtain
		$$
		x^*(t)=T_{\zeta^*}(t)x_0(\zeta^*)
		+\int_0^tT_{\zeta^*}(t-s)f(s)\,ds
		+\int_0^tT_{\zeta^*}(t-s)\mathcal Bw^*(s)\,ds.
		$$
		Hypotheses $H(\widetilde F)$(ii),(iv) and $H(\widetilde G)$(ii),(iv), together with \eqref{e6.18} and \eqref{e6.21}, imply a.e. convergence in Hausdorff distance of the corresponding values. The Aumann convergence theorem therefore yields
		$$
		f(t)\in F(t,x^*(t),u^*(t),\zeta^*),
		\qquad
		g(t)\in G(t,x^*(t),u^*(t),\eta^*)
		\quad\text{for a.e. }t\in I.
		$$
		Moreover, $\omega_k:=-\dot u_k+g_k\rightharpoonup
		\omega:=-\dot u^*+g$ in $L^2(I,H)$ by \eqref{e6.21}. Repeating the Mazur argument from Lemma \ref{l6.01}, we may work pointwise a.e. For such $t$, let
		$B_k:=B(t,u_k(t),\eta_k)$ and $B_*:=B(t,u^*(t),\eta^*)$. Then
		$$
		\operatorname{dis}(B_k,B_*)
		\leq\lambda\|u_k(t)-u^*(t)\|
		+\lambda'd_\Xi(\eta_k,\eta^*)\to0.
		$$
		The Minty argument based on Lemma \ref{l2.02}, used in the proof of Lemma \ref{l6.01}, now gives
		$-\dot u^*(t)+g(t)\in B(t,u^*(t),\eta^*)(u^*(t))$ a.e. Consequently,
		\begin{equation}\label{e6.22}
			(x^*,u^*)\in\mathcal S_{\zeta^*,\eta^*}(w^*).
		\end{equation}
		
		Set
		$q_{D_0}(t):=(1+a_0)\beta_{a_0}(t)
		+(1+a_0)\vartheta_{D_0}(t)$.
		Then $q_{D_0}\in L^1(I,\mathbb R_+)$ and
		\begin{equation}\label{e6.23}
			|\ell(t,x_k(t),u_k(t),\zeta_k,\eta_k)|
			\leq q_{D_0}(t)
			\quad\text{for a.e. }t\in I\text{ and all }k.
		\end{equation}
		Hypothesis $H(\ell)$(v), continuity of $\varphi$, and
		\eqref{e6.18}, \eqref{e6.21}, and \eqref{e6.23} yield, by dominated convergence,
		\begin{equation}\label{e6.24}
			\mathcal J_{\zeta_k,\eta_k}(x_k,u_k)
			\to\mathcal J_{\zeta^*,\eta^*}(x^*,u^*).
		\end{equation}
		Using \eqref{e6.22}, weak lower semicontinuity of $\mathcal R$, and
		\eqref{e6.19}, \eqref{e6.24}, we obtain
		\begin{align*}
			\mathcal F_{\rm joint}(w^*,\zeta^*,\eta^*)
			&\leq\mathcal J_{\zeta^*,\eta^*}(x^*,u^*)+\mathcal R(w^*)\\
			&\leq\liminf_{k\to\infty}
			\bigl(\mathcal J_{\zeta_k,\eta_k}(x_k,u_k)+\mathcal R(w_k)\bigr)\\
			&\leq\lim_{k\to\infty}\mathcal F_{\rm joint}(w_k,\zeta_k,\eta_k)
			=\inf_{\mathcal W_{ad}\times\mathcal P}\mathcal F_{\rm joint},
		\end{align*}
		where the last equality follows from \eqref{e6.18}. The reverse inequality follows from the admissibility in \eqref{e6.18}. Hence equality holds, and $(w^*,\zeta^*,\eta^*)$ solves Problem \ref{p6}.
	\end{proof}
	
	\begin{remark}
		Define $J(w,\zeta,\eta):=m_{\zeta,\eta}(w)+\mathcal R(w)$ and set
		\begin{equation}\label{e6.25}
			\begin{aligned}
				\alpha&:=\inf_{(w,\zeta,\eta)\in\mathcal W_{ad}\times\mathcal P}
				J(w,\zeta,\eta),\\
				\beta&:=\inf_{w\in\mathcal W_{ad}}
				\inf_{(\zeta,\eta)\in\mathcal P}J(w,\zeta,\eta).
			\end{aligned}
		\end{equation}
		Since $\alpha\leq J(w,\zeta,\eta)$ for every admissible triplet, taking the parameter infimum and then the control infimum gives $\alpha\leq\beta$. Conversely,
		$\inf_{(\zeta',\eta')\in\mathcal P}J(w,\zeta',\eta')
		\leq J(w,\zeta,\eta)$; taking first the control infimum and then the product-space infimum gives $\beta\leq\alpha$. Thus $\alpha=\beta$, and \eqref{e6.25} yields
		\begin{equation}\label{e6.26}
			\inf_{(w,\zeta,\eta)\in\mathcal W_{ad}\times\mathcal P}J(w,\zeta,\eta)
			=
			\inf_{w\in\mathcal W_{ad}}\inf_{(\zeta,\eta)\in\mathcal P}J(w,\zeta,\eta).
		\end{equation}
		Representation \eqref{e6.26} shows that Problem \ref{p6} may be viewed either as minimization over $\mathcal W_{ad}\times\mathcal P$ or as outer minimization in $w$ of the parameter-wise optimal value.
	\end{remark}
	
	The following auxiliary result extracts the common semicontinuity properties of the extremal value maps over a compact parameter set, used in both the robust and Hurwicz formulations below.
	
	\begin{proposition}\label{p6.02}
		Assume that the hypotheses of Proposition \ref{p6.01} hold, that hypothesis $H(0)$ is satisfied, and that the parameter set $\mathcal{P} \subset \Lambda \times \Xi$ is compact. Then:
		\begin{itemize}
			\item[\textup{(a)}] For every $w \in \mathcal{W}_{ad}$, the map $(\zeta,\eta) \mapsto m_{\zeta,\eta}(w)$ is continuous on $\Lambda \times \Xi$. In particular, the extrema
			\begin{equation}\label{e6.27}
			\varphi_{\inf}(w) := \inf_{(\zeta,\eta) \in \mathcal{P}} m_{\zeta,\eta}(w)
			\quad \text{and} \quad
			\varphi_{\sup}(w) := \sup_{(\zeta,\eta) \in \mathcal{P}} m_{\zeta,\eta}(w)
			\end{equation}
			are attained on $\mathcal{P}$.
			\item[\textup{(b)}] The map $\varphi_{\sup}: \mathcal{W}_{ad} \to \mathbb{R}$ is sequentially weakly lower semicontinuous.
			\item[\textup{(c)}] If, in addition, the map $(w,\zeta,\eta) \mapsto m_{\zeta,\eta}(w)$ is sequentially lower semicontinuous on $\mathcal{W}_{ad} \times \mathcal{P}$, where $\mathcal{W}_{ad}$ carries the weak topology of $L^2(I,Y)$ and $\mathcal{P}$ the metric topology, then $\varphi_{\inf}$ is also sequentially weakly lower semicontinuous.
		\end{itemize}
	\end{proposition}
	
	\begin{proof}
		(a) For a fixed $w \in \mathcal{W}_{ad}$, set $\widehat{F}_w(t,x,u,\zeta) := F(t,x,u,\zeta)+\mathcal{B}w(t)$. The translation by the fixed term $\mathcal{B}w(t)$ preserves measurability, weak compactness, convexity, and Hausdorff distances. The growth bound becomes $\|\widehat{F}_w(t,x,u,\zeta)\|_E \leq M_F(1+\|x\|_E+\|u\|)+\|\mathcal{B}\|_{\mathcal{L}(Y,E)}\phi_W(t)$, where the additional source $\|\mathcal{B}\|\phi_W \in L^2(I,\mathbb{R}_+)$ is absorbed into the Gr\"onwall estimate behind Theorem \ref{t4.01} without affecting its structure. Hence the parameter-dependent assumptions used in Theorem \ref{t5.01} remain valid, and $(\zeta,\eta) \mapsto m_{\zeta,\eta}(w)$ is continuous. Compactness of $\mathcal{P}$ then gives attainment of both extrema in \eqref{e6.27}.
		
		(b) Let $w_n \rightharpoonup w$ in $L^2(I,Y)$. For every $(\zeta,\eta) \in \mathcal{P}$, Proposition \ref{p6.01} yields
		\begin{equation}\label{e6.28}
		m_{\zeta,\eta}(w) \leq \liminf_{n \to \infty} m_{\zeta,\eta}(w_n)
		\leq
		\liminf_{n \to \infty} \varphi_{\sup}(w_n).
		\end{equation}
		Taking the supremum in \eqref{e6.28} and using \eqref{e6.27} gives
		$\varphi_{\sup}(w) \leq \liminf_{n \to \infty} \varphi_{\sup}(w_n)$.
		
		(c) Let $w_n\rightharpoonup w$ and set
		$\underline\varphi:=\liminf_{n\to\infty}\varphi_{\inf}(w_n)$.
		By part (a) and compactness of $\mathcal P$, after passing to a subsequence we can choose $(\zeta_n,\eta_n)\in\mathcal P$ such that
		\begin{equation}\label{e6.29}
			\varphi_{\inf}(w_n)=m_{\zeta_n,\eta_n}(w_n)\to\underline\varphi,
			\qquad
			(\zeta_n,\eta_n)\to(\zeta^*,\eta^*)\in\mathcal P.
		\end{equation}
		Joint sequential lower semicontinuity and \eqref{e6.29} give
		$m_{\zeta^*,\eta^*}(w)\leq\underline\varphi$. By \eqref{e6.27},
		$\varphi_{\inf}(w)\leq m_{\zeta^*,\eta^*}(w)$, which proves the claim.
	\end{proof}
	
	We next turn to a robust optimal control formulation. Unlike Problem \ref{p6}, where the parameters are chosen cooperatively, we now interpret $(\zeta, \eta)$ as uncertain but bounded quantities and optimize the control against the worst admissible parameter realization.
	
	Let $\mathcal{P} \subset \Lambda \times \Xi$ be a nonempty compact set representing the admissible uncertainty range. The robust problem takes the form of a min-max optimization problem:
	
	\begin{problem}\label{p7}
		Find a control $w^* \in \mathcal{W}_{ad}$ such that:
		$$
		\mathcal{F}_{rob}(w^*) = \inf_{w \in \mathcal{W}_{ad}} \mathcal{F}_{rob}(w),
		$$
		where
		$$
		\mathcal{F}_{rob}(w)
		:=
		\sup_{(\zeta, \eta) \in \mathcal{P}} \mathcal{F}_{joint}(w, \zeta, \eta)
		=
		\sup_{(\zeta, \eta) \in \mathcal{P}} \bigl\{ m_{\zeta, \eta}(w) + \mathcal{R}(w) \bigr\}.
		$$
	\end{problem}
	
	Since the regularization term $\mathcal{R}$ does not depend on the parameters, the preceding expression can also be written as
	$$
	\mathcal{F}_{rob}(w)
	=
	\left(\sup_{(\zeta, \eta) \in \mathcal{P}} m_{\zeta, \eta}(w)\right) + \mathcal{R}(w).
	$$
	
	\begin{theorem}\label{t6.03}
		Suppose that the hypotheses of Proposition \ref{p6.01} hold. Assume in addition hypothesis $H(0)$. If the parameter uncertainty set $\mathcal{P} \subset \Lambda \times \Xi$ is compact, then Problem \ref{p7} admits at least one optimal control $w^* \in \mathcal{W}_{ad}$.
	\end{theorem}
	
	\begin{proof}
		By Proposition \ref{p6.02}(a)--(b) and \eqref{e6.27}, $\varphi_{\sup}$ is finite-valued and sequentially weakly lower semicontinuous. By \eqref{e6.02}, the same is true of $\mathcal R$, and hence of $\mathcal F_{\rm rob}=\varphi_{\sup}+\mathcal R$.
		
		Let $\{w_n\}\subset\mathcal W_{ad}$ be a minimizing sequence. Weak compactness of $\mathcal W_{ad}$ gives, after passing to a subsequence,
		\begin{equation}\label{e6.30}
			\mathcal F_{\rm rob}(w_n)\to\inf_{w\in\mathcal W_{ad}}\mathcal F_{\rm rob}(w),
			\qquad
			w_n\rightharpoonup w^*\in\mathcal W_{ad}.
		\end{equation}
		Weak lower semicontinuity yields
		\begin{equation}\label{e6.31}
			\mathcal F_{\rm rob}(w^*)\leq\liminf_{n\to\infty}\mathcal F_{\rm rob}(w_n).
		\end{equation}
		By \eqref{e6.30}--\eqref{e6.31},
		$\mathcal F_{\rm rob}(w^*)\leq\inf_{\mathcal W_{ad}}\mathcal F_{\rm rob}$; the reverse inequality follows from $w^*\in\mathcal W_{ad}$. Thus $w^*$ solves Problem \ref{p7}.
	\end{proof}
	
	We finally introduce a Hurwicz-type compromise formulation, which interpolates between the optimistic design criterion and the pessimistic robust criterion. To balance these two extremes, we fix an optimism index $\alpha \in [0,1]$ and combine the best-case and worst-case marginal values over the uncertainty set.
	
	\begin{problem}\label{p8}
		Let $\alpha \in [0, 1]$ be a given optimism index. Find a compromise optimal control $w^*_\alpha \in \mathcal{W}_{ad}$ such that
		$$
		\mathcal{H}_\alpha(w^*_\alpha) = \inf_{w \in \mathcal{W}_{ad}} \mathcal{H}_\alpha(w),
		$$
		where the Hurwicz objective functional $\mathcal{H}_\alpha: \mathcal{W}_{ad} \to \mathbb{R}$ is defined as
		$$
		\mathcal{H}_\alpha(w) := \alpha \inf_{(\zeta, \eta) \in \mathcal{P}} m_{\zeta, \eta}(w) + (1-\alpha) \sup_{(\zeta, \eta) \in \mathcal{P}} m_{\zeta, \eta}(w) + \mathcal{R}(w).
		$$
	\end{problem}
	
	Since $\mathcal{R}$ does not depend on the uncertain parameters, the functional $\mathcal{H}_\alpha$ combines the optimistic marginal value $\inf_{(\zeta,\eta)\in\mathcal{P}} m_{\zeta,\eta}(w)$ and the pessimistic marginal value $\sup_{(\zeta,\eta)\in\mathcal{P}} m_{\zeta,\eta}(w)$ with the same control regularization term. In particular, the case $\alpha = 1$ recovers the iterated optimistic criterion associated with Problem \ref{p6}, while $\alpha = 0$ reduces to the robust objective from Problem \ref{p7}. Under the compactness and continuity assumptions used below, both extremal values are attained on $\mathcal{P}$ for every fixed control $w \in \mathcal{W}_{ad}$. We now establish the existence of optimal controls for this compromise problem.
	
	\begin{theorem}\label{t6.04}
		Suppose that the hypotheses of Proposition \ref{p6.01} hold. Assume in addition hypothesis $H(0)$, and assume that the map
		$$
		(w, \zeta, \eta) \mapsto m_{\zeta, \eta}(w)
		$$
		is sequentially weakly lower semicontinuous on $\mathcal{W}_{ad} \times \mathcal{P}$, where $\mathcal{W}_{ad}$ is endowed with the weak topology of $L^2(I, Y)$ and $\mathcal{P}$ carries the metric topology of $\Lambda \times \Xi$. If the parameter set $\mathcal{P} \subset \Lambda \times \Xi$ is compact, then for every optimism index $\alpha \in [0, 1]$, Problem \ref{p8} admits at least one optimal control $w^*_\alpha \in \mathcal{W}_{ad}$.
	\end{theorem}
	
	\begin{proof}
		By Proposition \ref{p6.02} and \eqref{e6.27}, both $\varphi_{\inf}$ and $\varphi_{\sup}$ are finite-valued and sequentially weakly lower semicontinuous. The regularization $\mathcal R$ has the same property by \eqref{e6.02}.
		
		For $w_n\rightharpoonup w$, the marginal values are uniformly bounded below by the majorant supplied by $H(\ell)$(ii),(iv). Since $\alpha,1-\alpha\geq0$, superadditivity of the limit inferior gives
		\begin{equation}\label{e6.32}
			\begin{aligned}
				\mathcal H_\alpha(w)
				&\leq\alpha\liminf_n\varphi_{\inf}(w_n)
				+(1-\alpha)\liminf_n\varphi_{\sup}(w_n)
				+\liminf_n\mathcal R(w_n)\\
				&\leq\liminf_{n\to\infty}\mathcal H_\alpha(w_n).
			\end{aligned}
		\end{equation}
		Thus $\mathcal H_\alpha$ is sequentially weakly lower semicontinuous. If $\{w_n\}\subset\mathcal W_{ad}$ is minimizing, weak compactness gives, after passing to a subsequence,
		\begin{equation}\label{e6.33}
			\mathcal H_\alpha(w_n)\to\inf_{w\in\mathcal W_{ad}}\mathcal H_\alpha(w),
			\qquad
			w_n\rightharpoonup w_\alpha^*\in\mathcal W_{ad}.
		\end{equation}
		Equations \eqref{e6.32}--\eqref{e6.33} imply
		$\mathcal H_\alpha(w_\alpha^*)\leq\inf_{\mathcal W_{ad}}\mathcal H_\alpha$.
		The reverse inequality follows from admissibility, so $w_\alpha^*$ solves Problem \ref{p8}.
	\end{proof}

	\section{Numerical Experiments and Computational Illustrations}\label{s10}
	
	We present numerical experiments that illustrate the qualitative behavior of the coupled systems and the performance of the semi-implicit feedback discretization underlying Theorem \ref{t4.01}. Each benchmark is a fixed-parameter instance of Problem \ref{p2a}. In the normal-cone setting $B=N_C$, the resolvent step reduces to metric projection onto the current constraint set, leading to the explicit updates used below.
	
	\begin{example}[Scalar Coupled System with State-Dependent Moving Barrier]\label{ex10.01}
		Consider a one-dimensional benchmark that retains state-dependent feedback while admitting an explicit projection formula. Let $E=H=\mathbb{R}$ and $I=[0,6]$. The dynamics are
		$$
		\begin{cases}
			\dot{x}(t) = -2x(t) + \frac{1}{2}\sin(u(t)) + \cos(t), & x(0) = 0.5, \\
			\dot{u}(t) \in -N_{C(t, u(t))}(u(t)) + x(t) - u(t), & u(0) = 0.
		\end{cases}
		$$
		The state-dependent moving set is the half-line
		$$
		C(t, u) := [\beta(t, u), +\infty), \quad \text{with barrier function } \beta(t, u) = 0.5\sin(3t) + 0.4 u.
		$$
		The term $0.5\sin(3t)$ produces an oscillating boundary, whereas the affine feedback coefficient $0.4<1$ satisfies the contraction condition required by the semi-implicit scheme.
	\end{example}

	\noindent\textbf{I. Structural Verification and Baseline Simulation for Example \ref{ex10.01}}
	
	\noindent \textit{Verification of the abstract hypotheses.}
	For $B(t,u):=N_{C(t,u)}$ with $C(t,u)=[\beta(t,u),+\infty)$, direct calculation gives
	$\operatorname{dis}(B(t,u),B(s,v))\le 1.5|t-s|+0.4|u-v|$,
	so $H(\widetilde{B})$(i) holds with $\theta(t)=1.5t$ and $\lambda=0.4<1$. For $H(\widetilde{B})$(iv), a scalar boundary--interior case distinction gives
	$\langle v_1-v_2,u_1-u_2\rangle\geq0$ for $v_i\in N_{C(t,u_i)}(u_i)$; thus one may take $\varrho_D\equiv0$. Since $E=\mathbb{R}$ is finite-dimensional, $T(t)=e^{-2t}$ is compact for $t>0$, and $H(\widetilde{A})$ holds with $M=1$. The continuous perturbations $F(t,x,u)=0.5\sin(u)+\cos(t)$ and $G(t,x,u)=x-u$ satisfy $H(\widetilde{F})$ and $H(\widetilde{G})$ with $M_F=1.5$ and $M_G=1$, respectively. Finally, $u_0=0\in C(0,0)$, and the remaining conditions follow directly in finite dimensions.
	
	\noindent\textit{Semi-implicit discretization.}
	With step size $\tau = 10^{-3}$ and grid points $t_i = i\tau$, the state dependence in the sweeping set is frozen at the previous iterate:
	$$
	x_{i} = x_{i-1} + \tau \left( -2x_{i-1} + 0.5\sin(u_{i-1}) + \cos(t_{i-1}) \right),
	\qquad
	u_i = \max\!\bigl\{\, u_{i-1} + \tau (x_i - u_{i-1}),\; 0.5\sin(3t_i) + 0.4 u_{i-1}\bigr\}.
	$$
	
	Figure \ref{fig:num:ex1-reference} displays alternating contact and free-motion regimes. The trajectory follows the constraint during rising boundary phases (approximately $t=0.2$--$0.6$, $2.2$--$2.7$, and $4.2$--$4.8$) and evolves above it during the intervening intervals. The projection step enforces feasibility at every time level.
	
	\begin{figure}[htbp]
		\centering
		\includegraphics[width=0.5\textwidth]{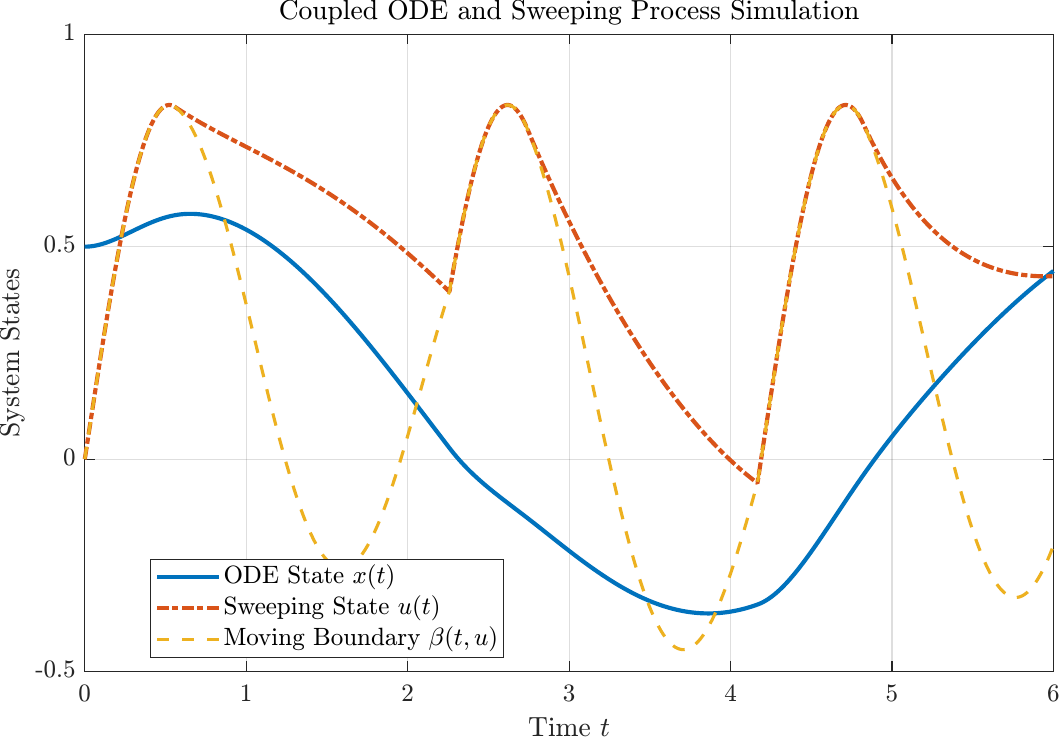}
		\caption{Baseline simulation for Example \ref{ex10.01}: the ODE state $x(t)$ (blue solid), sweeping state $u(t)$ (red dash-dotted), and moving boundary $\beta(t,u(t))$ (yellow dashed). The sweeping trajectory alternates between boundary contact and interior motion while remaining feasible.}
		\label{fig:num:ex1-reference}
	\end{figure}

	\noindent\textbf{II. Stability Under Parameter Perturbations for Example \ref{ex10.01}}
	
	To illustrate the trajectory stability associated with Theorem \ref{t4.02}, let $\epsilon_n\to0$ and consider
	\begin{equation}\label{e10.01}
		\begin{cases}
			\dot{x}_n(t) = -2x_n(t) + \frac{1}{2}\sin(u_n(t)) + \cos(t) + \zeta_n, & x_n(0) = 0.5 + \zeta_n, \\
			\dot{u}_n(t) \in -N_{C_n(t, u_n(t))}(u_n(t)) + x_n(t) - u_n(t), & u_n(0) = 0,
		\end{cases}
	\end{equation}
	with $\zeta_n=\epsilon_n$, $\eta_n=0.1\epsilon_n$, and
	$\beta_n(t,u)=0.5\sin(3t)+(0.4+\eta_n)u$. The resulting feedback coefficient
	$L_n=0.4+\eta_n$ remains below $1$ for all tested perturbations. We measure the relative trajectory deviation by
	$$
	E_n := \frac{\|x_n - x\|_{C(I)} + \|u_n - u\|_{C(I)}}{\|x\|_{C(I)} + \|u\|_{C(I)}}.
	$$
	
	In Figure \ref{fig:num:ex1-stability}, the log--log plot of $E_n$ has slope close to one, consistent with an approximately first-order response along the selected perturbation path. The other panels show the perturbed trajectories approaching the nominal solution as $\epsilon$ decreases through $\{0.15,0.08,0.02\}$.
	
	\begin{figure}[htbp]
		\centering
		\includegraphics[width=\textwidth]{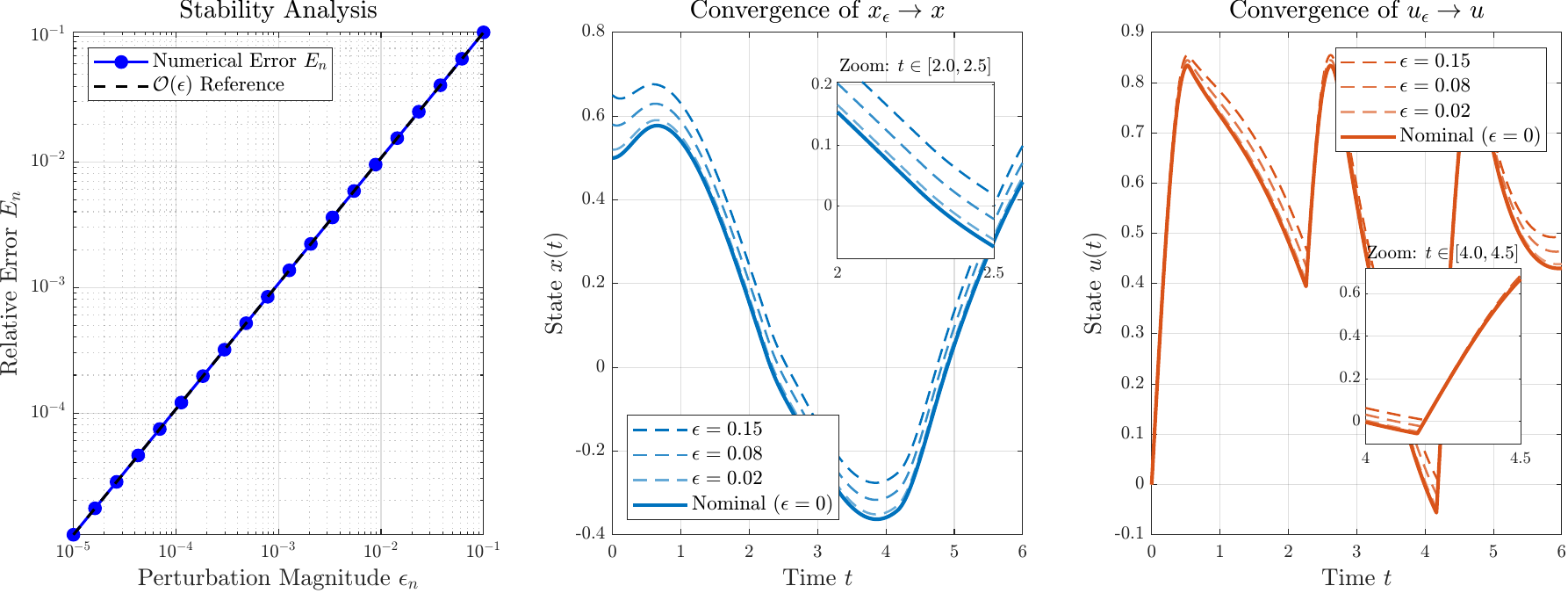}
		\caption{Stability experiment for Example \ref{ex10.01}. \textbf{Left:} Relative error $E_n$ versus $\epsilon_n\in[10^{-5},10^{-1}]$ on log--log axes, together with a reference line of slope $1$. \textbf{Center and right:} Perturbed states for $\epsilon\in\{0.15,0.08,0.02\}$ and the nominal trajectory $\epsilon=0$; the insets show deviations near active constraint intervals.}
		\label{fig:num:ex1-stability}
	\end{figure}
	
	\noindent\textbf{III. Sensitivity of the Trajectory Cost for Example \ref{ex10.01}}
	
	For the same perturbation family, consider the discounted Bolza cost
	$$
	\mathcal{J}(x, u) = \int_0^6 e^{-0.5 t} \bigl( \tfrac{1}{2}|x(t)|^2 + \tfrac{1}{2}|u(t)|^2 \bigr) dt + \tfrac{1}{2}|x(6)|^2.
	$$
	The nominal value is $V_{\mathrm{tr}}(0)\approx0.6357$. The observed discrepancy
	$\Delta V_\epsilon:=|V_{\mathrm{tr}}(\epsilon)-V_{\mathrm{tr}}(0)|$ scales approximately as $\mathcal O(\epsilon)$, consistent with the transfer of trajectory stability to the cost.
	
\noindent\textbf{IV. Fixed-Parameter Control Experiment for Example \ref{ex10.01}}
	
	We introduce a bounded scalar control $w(t)\in[-1,1]$ in the $x$-equation and minimize
	\begin{equation}\label{e10.02}
		\mathcal{J}(w) := \frac{1}{2} \int_0^6 \left( |x(t)|^2 + |u(t)|^2 + 0.1 |w(t)|^2 \right) dt + \frac{1}{2} \left( |x(6)|^2 + |u(6)|^2 \right).
	\end{equation}
	Figure \ref{fig:num:ex1-control} shows the computed control and state response. The control remains within $[-1,1]$ and primarily attenuates $x(t)$, with a smaller effect on $u(t)$. The projected trajectory remains feasible. The objective decreases from $J_{\mathrm{nom}}\approx1.5293$ to $J^*\approx1.0121$, a reduction of approximately $33.82\%$.
	
	\begin{figure}[htbp]
		\centering
		\includegraphics[width=\textwidth]{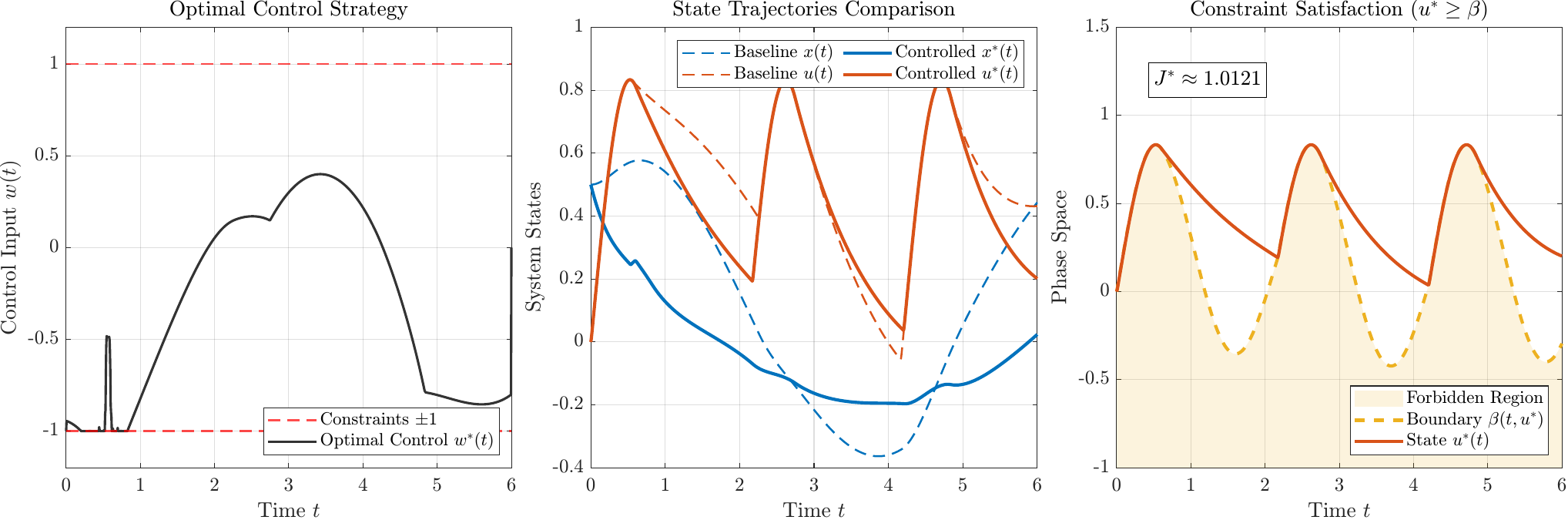}
		\caption{Fixed-parameter control for Example \ref{ex10.01}. \textbf{Left:} Admissible control $w^*(t)$. \textbf{Center:} Uncontrolled (dashed) and controlled (solid) trajectories. \textbf{Right:} Feasibility of $u^*(t)$ relative to $\beta(t,u^*(t))$; the computed cost is $J^*\approx1.0121$.}
		\label{fig:num:ex1-control}
	\end{figure}
	
\noindent\textbf{V. Joint Design--Control Optimization for Example \ref{ex10.01}}
	
	We next optimize $w(t)$ jointly with
	$(\zeta,\eta)\in\mathcal P:=[0,0.25]\times[-0.15,0.15]$ subject to
	\begin{equation}\label{e10.03}
		\begin{cases}
			\dot{x}(t) = -2x(t) + \frac{1}{2}\sin(u(t)) + \cos(t) + w(t) + \zeta, \\
			\dot{u}(t) \in -N_{C_\eta(t, u(t))}(u(t)) + x(t) - u(t),
		\end{cases}
	\end{equation}
	where $C_\eta(t,u)=[\beta_\eta(t,u),+\infty)$ and
	$\beta_\eta(t,u)=0.5\sin(3t)+(0.4+\eta)u$. The computed solution in Figure \ref{fig:num:ex1-joint} has $\zeta^*\approx0$ and $\eta^*\approx-0.15$. The objective decreases from $J_{\mathrm{nom}}\approx1.5283$ to $J^*\approx0.6945$, a reduction of approximately $54.55\%$. Because $\zeta^*$ remains near zero, most of the improvement is associated with the boundary-feedback parameter $\eta$.
	
	\begin{figure}[htbp]
		\centering
		\includegraphics[width=\textwidth]{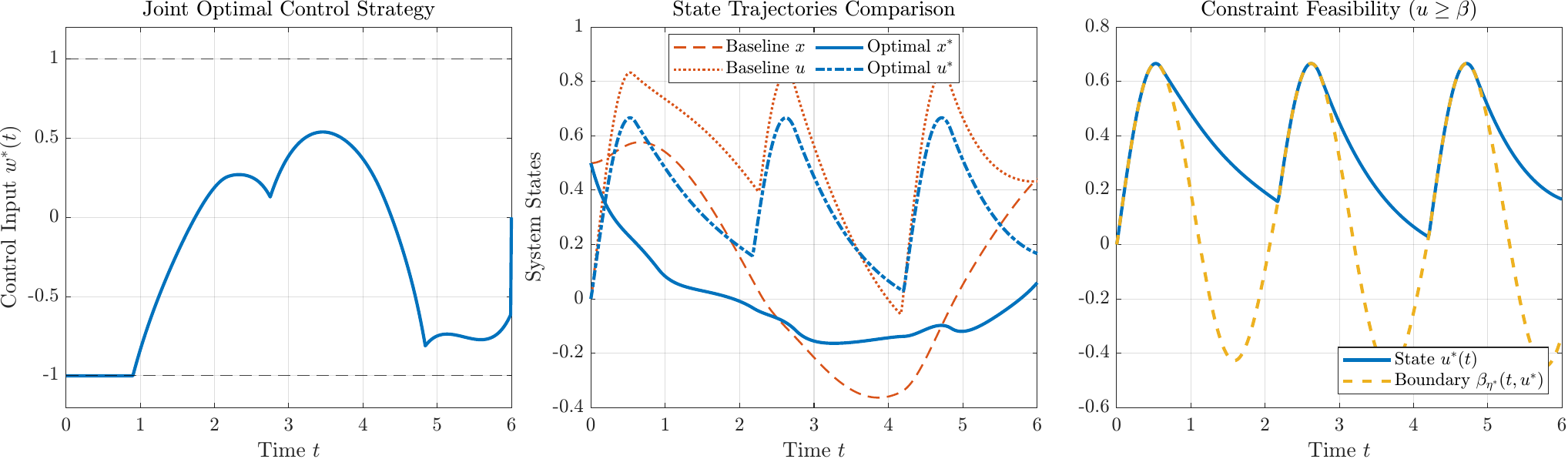}
		\caption{Joint design--control experiment for Example \ref{ex10.01}. \textbf{Left:} Optimized control. \textbf{Center:} Nominal and jointly optimized trajectories. \textbf{Right:} Feasibility relative to $\beta_{\eta^*}(t,u^*)$. The computed values are $\zeta^*\approx0$, $\eta^*\approx-0.15$, and $J^*\approx0.6945$.}
		\label{fig:num:ex1-joint}
	\end{figure}
	
	\noindent\textbf{VI. Robust Min--Max Optimal Control for Example \ref{ex10.01}}
	
	For the dynamics \eqref{e10.03} and uncertainty set
	$\mathcal P=[0,0.25]\times[-0.15,0.15]$, the computed inner maximizer is
	$(\zeta_{\mathrm{wc}},\eta_{\mathrm{wc}})=(0.25,0.15)$. At this adverse pair, the nominal-controller cost is $J_{\mathrm{nom,wc}}\approx1.8726$, whereas the robust controller gives $J_{\mathrm{rob,wc}}\approx1.7718$, an improvement of approximately $5.38\%$. Figure \ref{fig:num:ex1-robust} also displays the price-of-robustness trade-off: the cost curves cross near $\zeta=0.06$--$0.07$, after which the robust controller performs better. The principal trajectory-level change under the adverse scenario is the attenuation of $x(t)$.
	
	\begin{figure}[htbp]
		\centering
		\includegraphics[width=\textwidth]{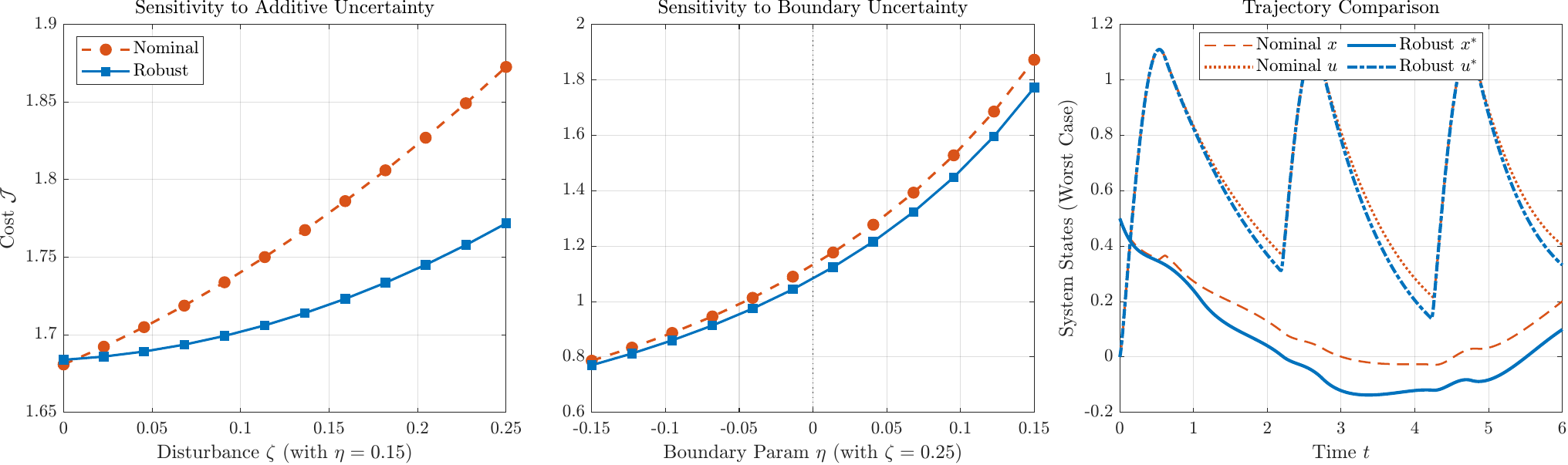}
		\caption{Robust-control experiment for Example \ref{ex10.01}. \textbf{Left:} Cost versus $\zeta$ at $\eta=0.15$. \textbf{Center:} Cost versus $\eta$ at $\zeta=0.25$. \textbf{Right:} Trajectories for the computed adverse pair $(0.25,0.15)$. Robust optimization reduces the adverse cost from approximately $1.8726$ to $1.7718$.}
		\label{fig:num:ex1-robust}
	\end{figure}
	
	\noindent\textbf{VII. Hurwicz Compromise Control for Example \ref{ex10.01}}
	
	Finally, for $\alpha\in\{0,0.1,0.5,0.9,1\}$, we minimize
	$$
	\mathcal H_\alpha(w)
	:=\alpha\inf_{(\zeta,\eta)\in\mathcal P}\mathcal J(w,\zeta,\eta)
	+(1-\alpha)\sup_{(\zeta,\eta)\in\mathcal P}\mathcal J(w,\zeta,\eta).
	$$
	Figure \ref{fig:num:ex1-hurwicz} shows that the best-case cost decreases from approximately $0.7967$ at $\alpha=0$ to $0.6983$ at $\alpha=1$, while the worst-case cost increases from approximately $1.7822$ to $1.9320$. The control profiles vary smoothly between the conservative and optimistic endpoints, with the largest visible change occurring between the balanced and highly optimistic cases.
	
	\begin{figure}[htbp]
		\centering
		\includegraphics[width=\textwidth]{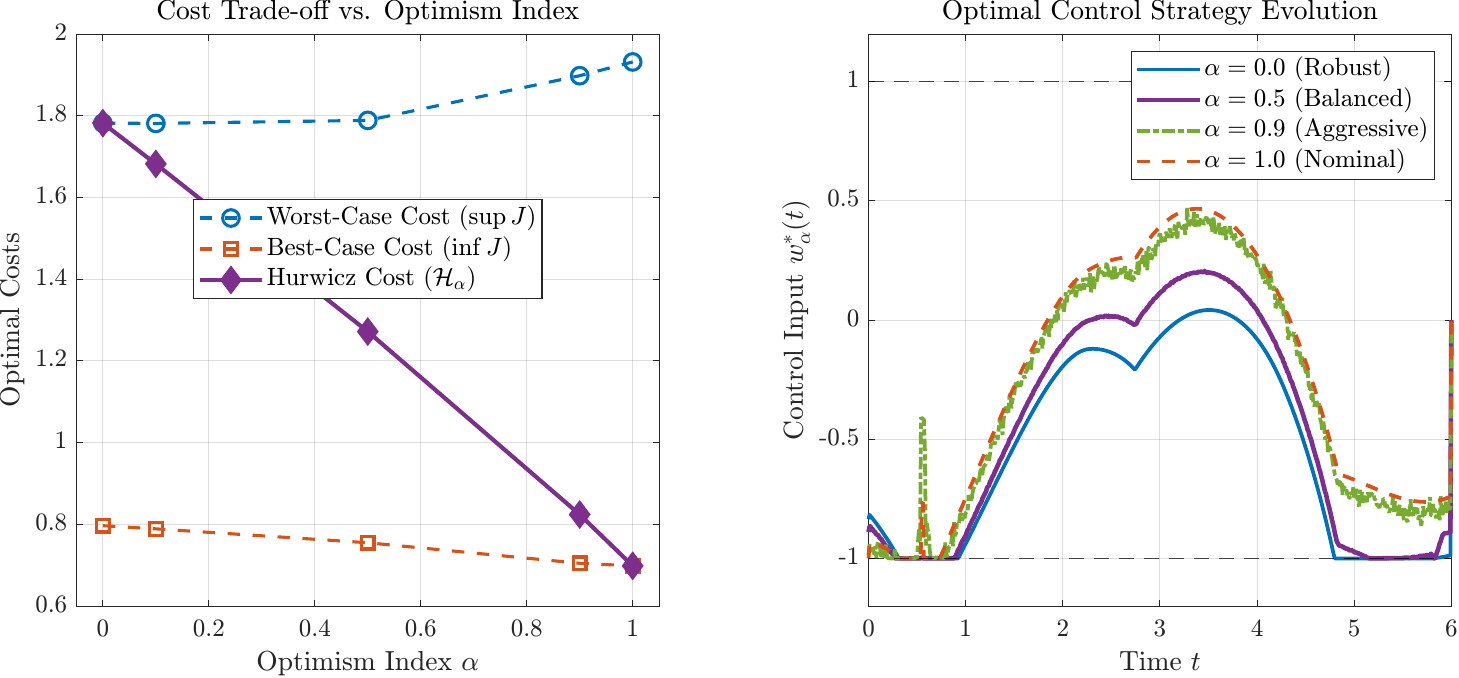}
		\caption{Hurwicz-control experiment for Example \ref{ex10.01}. \textbf{Left:} Best-case, worst-case, and Hurwicz values for $\alpha\in\{0,0.1,0.5,0.9,1\}$. \textbf{Right:} Controls at representative optimism levels; the profiles for $\alpha=0.9$ and $\alpha=1$ are nearly indistinguishable.}
		\label{fig:num:ex1-hurwicz}
	\end{figure}
	
	\begin{example}[Multidimensional Coupled System with State-Dependent Ball Constraint]\label{ex10.02}
		Consider a multidimensional system with $E=\mathbb{R}^2$, $H=\mathbb{R}^3$, and $I=[0,6]$:
		$$
		\begin{cases}
			\dot{x}(t) = A x(t) + F(t, x(t), u(t)), & x(0) = [1, -1]^\top, \\
			\dot{u}(t) \in -N_{C(t, u(t))}(u(t)) + G(t, x(t), u(t)), & u(0) = [0, 0, 0]^\top.
		\end{cases}
		$$
		The matrix $A$ and coupling terms $F$ and $G$ are
		$$
		A = \begin{pmatrix} -1 & 1 \\ -1 & -1 \end{pmatrix}, \quad 
		F(t, x, u) = \begin{pmatrix} 0.5 u_1 + 0.1 x_2 \cos(t) \\ 0.5 u_2 - 0.1 x_1 \sin(t) \end{pmatrix}, \quad 	G(t, x, u) = \begin{pmatrix} x_1 - 0.2 u_1 \\ x_2 - 0.2 u_2 \\ -5 + 0.1 u_3 \end{pmatrix}.
		$$
		The moving set is the closed unit ball centered at $\gamma(t,u)$:
		$$
		C(t, u) := \overline{\mathbb{B}}(\gamma(t, u), 1) = \left\{ v \in \mathbb{R}^3 \mid \|v - \gamma(t, u)\| \le 1 \right\},
		$$
		where
		$$
		\gamma(t, u) = \begin{pmatrix} \sin(t) \\ \cos(t) \\ t/2 \end{pmatrix} + 0.3 \begin{pmatrix} u_2 \\ u_3 \\ u_1 \end{pmatrix}.
		$$
	\end{example}
	
	\noindent\textbf{I. Structural Verification and Baseline Simulation for Example \ref{ex10.02}}
	
	\noindent \textit{Verification of the abstract hypotheses.}
	Write $\gamma(t,u)=\alpha(t)+0.3Pu$, where
	$\alpha(t)=(\sin t,\cos t,t/2)^\top$ and $P$ is the cyclic permutation matrix. For $B(t,u):=N_{C(t,u)}$, one obtains
	$\operatorname{dis}(B(t,u),B(s,v))\le 1.2|t-s|+0.3\|u-v\|$,
	so $H(\widetilde{B})$(i) holds with $\theta(t)=1.2t$ and $\lambda=0.3<1$. The standard hypomonotonicity estimate for normal cones to Lipschitz translations of a ball gives
	$\langle v_1-v_2,u_1-u_2\rangle\geq-\varrho_D(t)\|u_1-u_2\|^2$
	on bounded trajectory sets for a suitable $\varrho_D\in L^1$, which verifies $H(\widetilde{B})$(iv). The semigroup generated by $A$ is $T(t)=e^{-t}R(t)$, with $R(t)$ a rotation matrix, and satisfies $H(\widetilde{A})$ with $M=1$. The continuous maps $F$ and $G$ have linear growth, so $H(\widetilde{F})$ and $H(\widetilde{G})$ hold. Finally, $\|u_0-\gamma(0,0)\|=1$, and the remaining assumptions follow directly in finite dimensions.
	
	\noindent\textit{Semi-implicit discretization.}
	With $\tau=10^{-3}$, the update is
	$$
	\begin{aligned}
		x_i&=x_{i-1}+\tau\bigl(Ax_{i-1}+F(t_{i-1},x_{i-1},u_{i-1})\bigr),\\
		\widetilde u_i&=u_{i-1}+\tau G(t_{i-1},x_i,u_{i-1}),\\
		u_i&=\operatorname{proj}_{\overline{\mathbb B}(\gamma(t_i,u_{i-1}),1)}
		(\widetilde u_i).
	\end{aligned}
	$$
	
	Figure \ref{fig:num:ex2-reference} shows bounded oscillations in $x_1$ and $x_2$, consistent with the dissipative matrix $A$. The sweeping components respond more strongly to the moving geometry. In the three-dimensional view, $u(t)$ tracks the translating ball and moves upward toward the terminal feasible set.
	
	\begin{figure}[htbp]
		\centering
		\begin{minipage}[c]{0.48\linewidth}
			\centering
			\includegraphics[width=\linewidth]{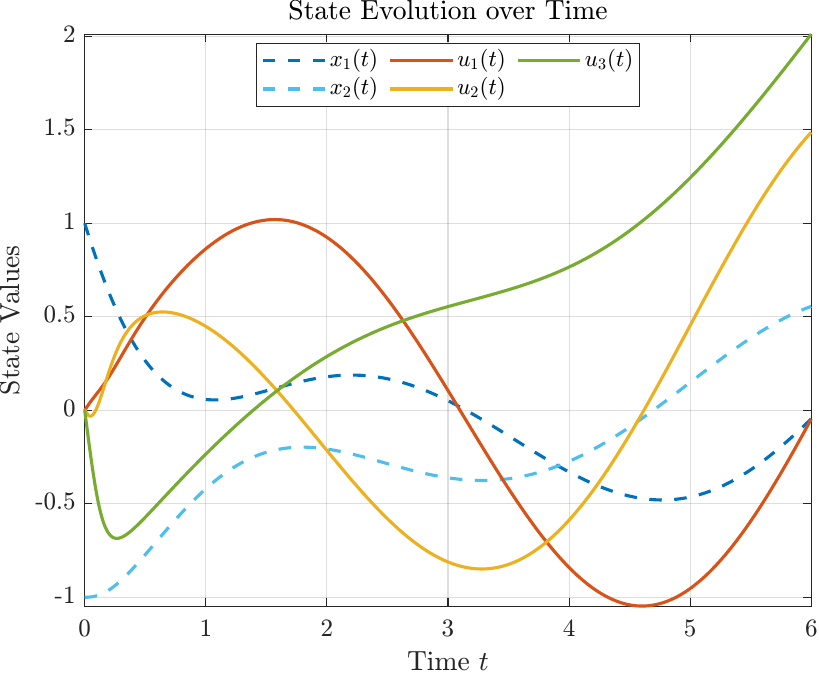}
		\end{minipage}
		\hfill
		\begin{minipage}[c]{0.48\linewidth}
			\centering
			\includegraphics[width=0.6\linewidth]{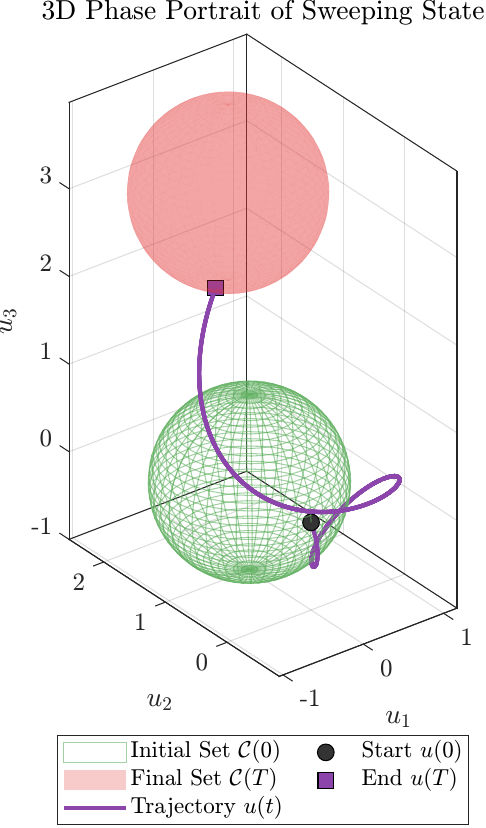}
		\end{minipage}
		\caption{Baseline simulation for Example \ref{ex10.02}. \textbf{Left:} ODE states $x_1,x_2$ (dashed) and sweeping states $u_1,u_2,u_3$ (solid). \textbf{Right:} Sweeping trajectory $u(t)$, initial ball $C(0,u(0))$, terminal ball $C(T,u(T))$, and endpoint markers.}
		\label{fig:num:ex2-reference}
	\end{figure}

	\noindent\textbf{II. Stability Under Coupled Perturbations for Example \ref{ex10.02}}
	
	We replace $A$ by $A_n=A+\epsilon_n I_{2\times2}$ and the center by
	$\gamma_n(t,u)=(\sin((1+\epsilon_n)t),\cos((1+\epsilon_n)t),t/2)^\top+0.3Pu$,
	thereby perturbing both the ODE and the moving geometry. Since
	$T_n(t)=e^{\epsilon_n t}T(t)$,
	\[
	\|T_n(t)-T(t)\|_{\mathcal L(\mathbb R^2)}
	=(e^{\epsilon_n t}-1)\|T(t)\|\to0
	\]
	uniformly on $I$. Thus the perturbed family satisfies the required semigroup continuity. In Figure \ref{fig:num:ex2-stability}, the log--log error curve has slope close to one for $\epsilon_n\in[10^{-4},10^{-1}]$, and all five state components approach the nominal solution as $\epsilon$ decreases.
	
	\begin{figure}[htbp]
		\centering
		\includegraphics[width=\textwidth]{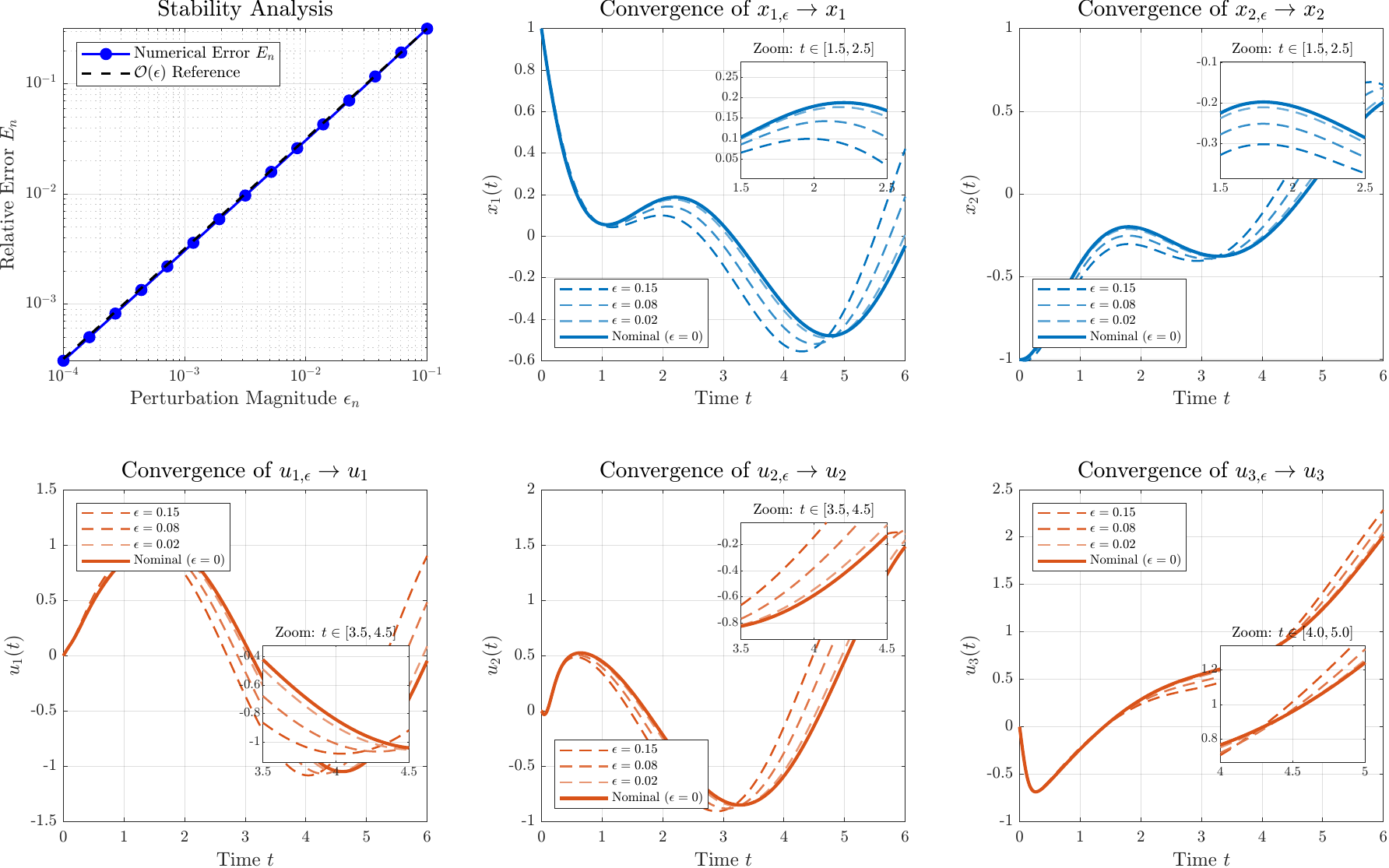}
		\caption{Stability experiment for Example \ref{ex10.02}. \textbf{Top left:} Relative error $E_n$ versus $\epsilon_n\in[10^{-4},10^{-1}]$ on log--log axes, with a slope-one reference. \textbf{Other panels:} Nominal and perturbed components for $\epsilon\in\{0.15,0.08,0.02\}$; the insets show intervals of rapid variation.}
		\label{fig:num:ex2-stability}
	\end{figure}
	
	\noindent\textbf{III. Sensitivity of the Trajectory Cost for Example \ref{ex10.02}}
	
	For the same perturbation family, consider
	\[
	\mathcal J(x,u)
	=\int_0^6e^{-0.5t}\left(\tfrac12\|x(t)\|^2+\tfrac12\|u(t)\|^2\right)dt
	+\tfrac12\|x(6)\|^2.
	\]
	The nominal value is $V_{\mathrm{tr}}(0)\approx1.5680$, and
	$\Delta V_\epsilon$ scales approximately as $\mathcal O(\epsilon)$. This behavior is consistent with the trajectory stability observed above.
	
	\noindent\textbf{IV. Fixed-Parameter Control for Example \ref{ex10.02}}
	
	We introduce a vector control $w(t)\in[-1,1]^2$ in the $x$-equation and minimize
	\begin{equation}\label{e10.04}
		\mathcal{J}(w) := \frac{1}{2} \int_0^6 \left( \|x(t)\|^2 + \|u(t)\|^2 + 0.1 \|w(t)\|^2 \right) dt + \frac{1}{2} \left( \|x(6)\|^2 + \|u(6)\|^2 \right).
	\end{equation}
	Figure \ref{fig:num:ex2-control} shows the computed control and trajectories. The objective decreases from $J_{\mathrm{nom}}\approx9.1690$ to $J^*\approx8.1512$, a reduction of approximately $11.10\%$. Most of the improvement comes from reducing $\|x(t)\|$; the effect on $\|u(t)\|$ is indirect through the coupling.
	
	\begin{figure}[htbp]
		\centering
		\includegraphics[width=\textwidth]{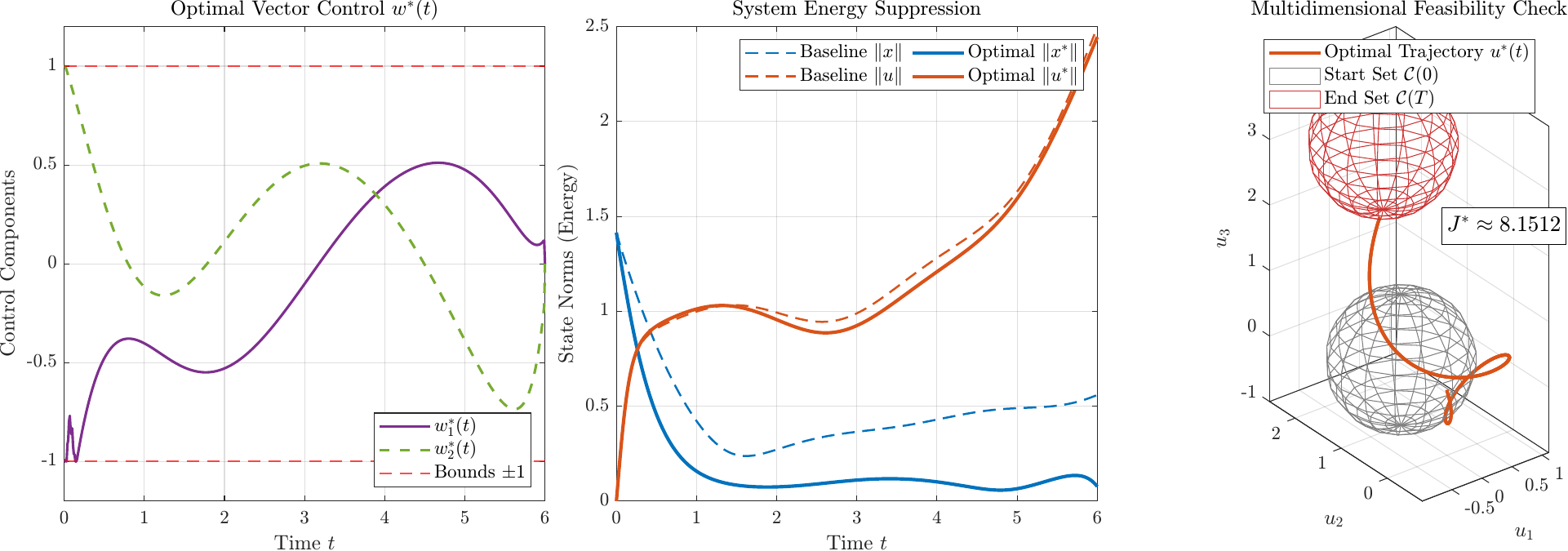}
		\caption{Fixed-parameter control for Example \ref{ex10.02}. \textbf{Left:} Control components under box constraints. \textbf{Center:} Uncontrolled (dashed) and controlled (solid) state norms. \textbf{Right:} Controlled sweeping trajectory with the initial and terminal balls; $J^*\approx8.1512$.}
		\label{fig:num:ex2-control}
	\end{figure}
	
	\noindent\textbf{V. Joint Design--Control Optimization in a Reduced Vector Setting}
	
	For a computationally tractable joint-design experiment, we use a two-dimensional componentwise extension of the scalar benchmark. Here
	$x(t),u(t),w(t)\in\mathbb R^2$, the design variables satisfy
	$\zeta\in[0,0.25]$ and $\eta\in[-0.15,0.15]$, and $e=(1,1)^\top$. The dynamics are
	\begin{equation}\label{e10.05}
		\begin{cases}
			\dot{x}(t) = -2x(t) + \frac{1}{2}\sin(u(t)) + \cos(t)e + w(t) + \zeta e, \\
			\dot{u}(t) \in -N_{C_\eta(t, u(t))}(u(t)) + x(t) - u(t), \\
			x(0) = (1,-1)^\top,\qquad u(0) = (0,0)^\top,
		\end{cases}
	\end{equation}
	where the sine is applied componentwise and
	\[
	C_\eta(t,u)=\prod_{i=1}^2[\beta_{\eta,i}(t,u),+\infty),\qquad
	\beta_\eta(t,u)=0.5\sin(3t)e+(0.4+\eta)u.
	\]
	The controls satisfy $|w_i(t)|\leq1$ a.e., and
	$\mathcal P=[0,0.25]\times[-0.15,0.15]$.
	
	We minimize \eqref{e10.04}, interpreted with Euclidean norms on $\mathbb R^2$, by optimizing the open-loop control samples and $(\zeta,\eta)$ directly in finite dimensions. Relative to $(w,\zeta,\eta)=(0,0,0)$, the computed solution reduces the objective from $J_{\mathrm{nom}}\approx3.0544$ to $J^*\approx1.4798$, or approximately $51.55\%$, and selects
	$\zeta^*\approx0$ and $\eta^*=-0.15$. Thus the boundary-feedback parameter accounts for most of the design improvement, whereas the additive shift remains near zero. Figure \ref{fig:num:ex2-joint} shows a marked reduction in $\|x(t)\|$ while both sweeping components remain feasible and contact their lower barriers during active phases.
	
	\begin{figure}[htbp]
		\centering
		\includegraphics[width=\textwidth]{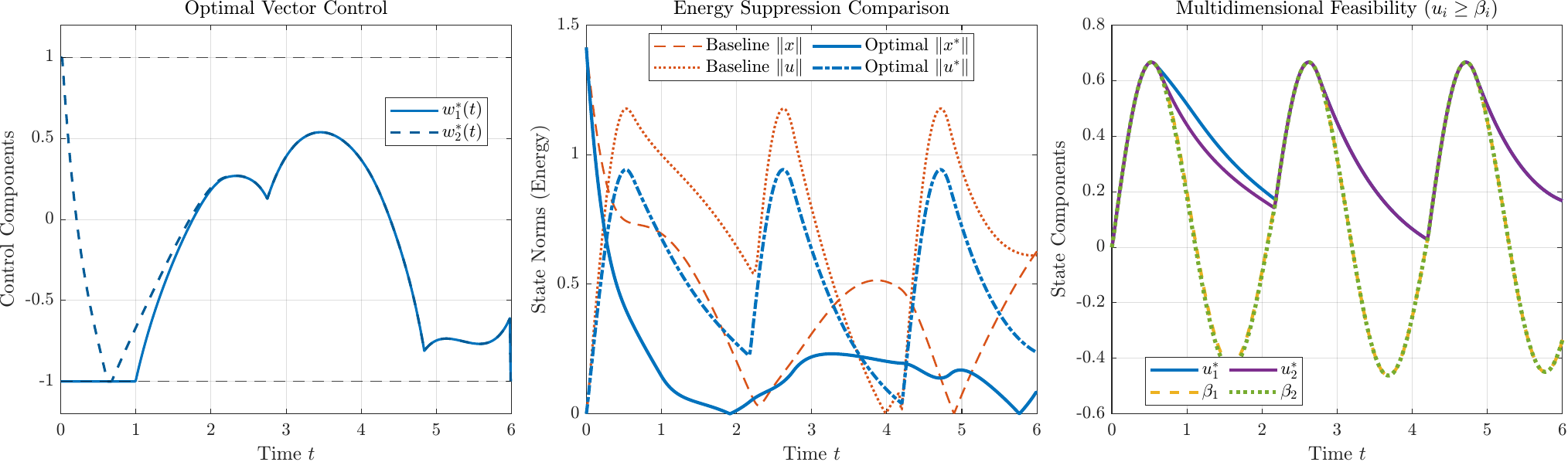}
		\caption{Joint design--control experiment for the reduced vector benchmark. \textbf{Left:} Optimized control components. \textbf{Center:} Uncontrolled and optimized state norms. \textbf{Right:} Sweeping components and their moving lower barriers. The computed values are $J_{\mathrm{nom}}\approx3.0544$, $J^*\approx1.4798$, $\zeta^*\approx0$, and $\eta^*=-0.15$.}
		\label{fig:num:ex2-joint}
	\end{figure}
	
	\noindent\textbf{VI. Robust Min--Max Control for Example \ref{ex10.02}}
	
	We return to the ball-constrained model and introduce additive and geometric uncertainty:
	\[
	\begin{aligned}
		\dot x(t)&=Ax(t)+F(t,x(t),u(t))+w(t)+\zeta(1,1)^\top,\\
		\dot u(t)&\in-N_{C_\eta(t,u(t))}(u(t))+G(t,x(t),u(t)),\\
		C_\eta(t,u)&=\overline{\mathbb B}(\gamma_\eta(t,u),1),\\
		\gamma_\eta(t,u)&=
		\bigl(\sin((1+\eta)t),\cos((1+\eta)t),t/2\bigr)^\top+0.3Pu.
	\end{aligned}
	\]
	The uncertainty set is $[0,0.5]\times[-0.2,0.2]$.
	
	Figure \ref{fig:num:ex2-robust} reports a direct discrete min--max computation with a blockwise-constant open-loop control. The computed inner maximizer is
	$(\zeta_{\mathrm{wc}},\eta_{\mathrm{wc}})=(0,0.2)$. At this pair, the nominal-controller cost is $J_{\mathrm{nom,wc}}\approx11.1412$, whereas the robust controller gives $J_{\mathrm{rob,wc}}\approx10.5764$, a reduction of approximately $5.07\%$. The cost curves cross near $\zeta=0.03$--$0.04$, after which the robust controller performs better. In this benchmark, the adverse response is governed mainly by the geometric parameter $\eta$, and the principal trajectory-level improvement is a reduction in $\|x(t)\|$.
	
	\begin{figure}[htbp]
		\centering
		\includegraphics[width=\textwidth]{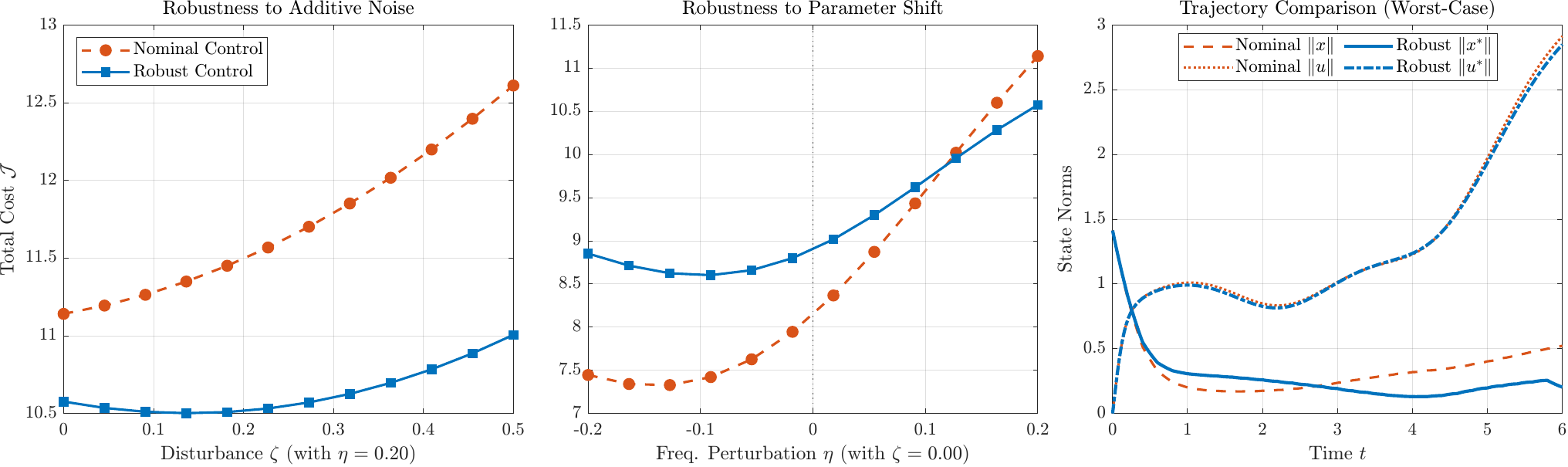}
		\caption{Robust-control experiment for the ball-constrained system. \textbf{Left:} Cost versus $\zeta$ at $\eta=0.2$. \textbf{Center:} Cost versus $\eta$ at $\zeta=0$. \textbf{Right:} State norms for the computed adverse pair $(0,0.2)$. Robust optimization reduces the cost from approximately $11.1412$ to $10.5764$.}
		\label{fig:num:ex2-robust}
	\end{figure}

	\section{Conclusion and Outlook}\label{s11}
	
	This paper investigated sensitivity and optimal control for parameter-dependent coupled evolution inclusions governed by state-dependent maximal monotone operators. The semi-implicit feedback construction underlying the analysis provides a common basis for the existence, compactness, and approximation arguments used throughout the paper. The main contributions can be summarized as follows.
	
	\begin{enumerate}[{\rm (i)}]
		\item \textbf{Analytical foundations and parameter stability.} The analysis builds on the existence, compactness, and Painlev\'e--Kuratowski continuity results established in the companion paper \cite{zengduwang2026submitted}. These properties provide the compactness and stability mechanisms required for the optimization results developed here.
		
		\item \textbf{Sensitivity of optimal values and solutions.} For Bolza-type objectives, we established the existence of optimal pairs, continuity of the value function, and upper semicontinuity of the optimal solution map. Thus, parameter stability is obtained at the levels of both feasible trajectories and optimal responses.
		
		\item \textbf{Control under multiple decision criteria.} After introducing an external control, we proved existence results for fixed-parameter optimal control, joint design--control optimization, robust min--max control, and Hurwicz-type compromise control. The latter places optimistic and pessimistic criteria within a single framework based on the weak lower semicontinuity of marginal value functions.
		
		\item \textbf{Numerical evidence.} The scalar and multidimensional sweeping benchmarks illustrate alternating active and inactive constraint regimes, the empirical stability of trajectories and costs under parameter perturbations, and the behavior of the four control formulations.
	\end{enumerate}
	
	Several questions remain open.
	
	\begin{enumerate}[{\rm (i)}]
		\item \textbf{Quantitative stability.} The present analysis yields Painlev\'e--Kuratowski continuity of the solution sets. Bounded-Hausdorff or Attouch--Wets estimates could strengthen this qualitative result and provide explicit perturbation bounds for trajectories, optimal values, and optimal solution sets.
		
		\item \textbf{Nonsmooth optimality conditions and higher-order sensitivity.} Because the robust and Hurwicz formulations involve infimal and supremal marginal functions, existence alone does not characterize their optimizers. Further analysis should develop generalized subdifferential representations and Danskin-type formulas for the marginal costs, leading to first- and second-order optimality conditions.
		
		\item \textbf{Discretization error and broader models.} Although the experiments in Section \ref{s10} support the numerical effectiveness of the semi-implicit feedback scheme, convergence rates for the fully discrete method remain to be established. Other directions include second-order sweeping processes, more general state-dependent constraints, and systems coupled with partial differential equations.
	\end{enumerate}

	\bibliographystyle{siam}
	\bibliography{refs}

\end{document}